\documentclass{amsart}

\calclayout

\usepackage{amsmath}
\usepackage{amssymb}
\usepackage[all]{xy}
\usepackage{enumerate}
\usepackage{color}

\def\Z{{\mathbb Z}}
\def\Q{{\mathbb Q}}
\def\C{{\mathbb C}}
\def\P{{\mathbb P}}

\def\H{{\mathbb H}}
\def\L{{\mathbb L}}
\def\V{{\mathbb V}}

\def\A{{\mathcal A}}

\def\cC{{\mathcal C}}
\def\cD{{\mathcal D}}

\def\cG{{\mathcal G}}
\def\cH{{\mathcal H}}

\def\M{{\mathcal M}}

\def\O{{\mathcal O}}
\def\cP{{\mathcal P}}

\def\T{{\mathcal T}}
\def\U{{\mathcal U}}
\def\cV{{\mathcal V}}
\def\X{{\mathcal X}}

\def\G{\Gamma}

\def\d{{\mathfrak d}}
\def\g{{\mathfrak g}}
\def\h{{\mathfrak h}}

\def\n{{\mathfrak n}}
\def\p{{\mathfrak p}}
\def\r{{\mathfrak r}}
\def\s{{\mathfrak s}}

\def\u{{\mathfrak u}}
\def\v{{\mathfrak v}}

\def\fX{{\mathfrak X}}
\def\etabar{{\overline{\eta}}}

\def\Ql{{\Q_\ell}}
\def\Zl{{\Z_\ell}}

\def\Gm{{\mathbb{G}_m}}
\def\Sp{{\mathrm{Sp}}}

\def\GSp{{\mathrm{GSp}}}

\def\arith{\mathrm{arith}}
\def\alg{\mathrm{alg}}
\def\geom{\mathrm{geom}}

\def\cts{\mathrm{cts}}

\def\orb{\mathrm{orb}}

\def\ab{\mathrm{ab}}

\def\et{\mathrm{\acute{e}t}}

\def\hyp{\mathrm{hyp}}

\def\D{\Delta}

\newcommand\id{\operatorname{id}}
\newcommand\ad{\operatorname{ad}}

\newcommand\Hom{\operatorname{Hom}}

\newcommand\Spec{\operatorname{Spec}}
\newcommand\Diff{\operatorname{Diff}}
\newcommand\Aut{\operatorname{Aut}}

\newcommand\Der{\operatorname{Der}}

\newcommand\Gr{\operatorname{Gr}}

\newcommand\Pic{\operatorname{Pic}}

\newtheorem{theorem}{Theorem}[section]
\newtheorem{lemma}[theorem]{Lemma}
\newtheorem{proposition}[theorem]{Proposition}
\newtheorem{corollary}[theorem]{Corollary}
\newtheorem{bigtheorem}{Theorem}
\newtheorem{bigcorollary}[bigtheorem]{Corollary}

\theoremstyle{definition}

\theoremstyle{remark}
\newtheorem{remark}[theorem]{Remark}

\begin{document}
 	
\title{Remarks on the sections of universal hyperelliptic curves }

\author{Tatsunari Watanabe}
\address{Mathematics Department, Embry-Riddle Aeronautical University, Prescott, AZ 86301}
\email{watanabt@erau.edu}

\maketitle
\begin{abstract}
 In this paper, we study the obstruction for the sections of the universal hyperelliptic curves of genus $g\geq 3$. 
The obstruction of our interest comes from the relative completion of the hyperelliptic mapping class groups and the Lie algebra of the unipotent completion of the fundamental group of the configuration space of a compact oriented surface. Using the obstruction, we prove that the Birman exact sequence for the hyperelliptic mapping class groups does not split for $g\geq 3$. 
\end{abstract}

\tableofcontents

%%%%%%%%%%%%%%%%%%%%%%%%%%%%%%%%%%%%%%%%%%%%%%%%%%%%%%%%%%%%%%%%%%%%%%%%%%%%%%%%%%%%%%%%%%%%%%%%%%%%%%%%%%%%%%%%
\section{Introduction}
In \cite{EaKr}, Earle and Kra proved that the Teichm\"uller  curve over the Teichm\"uller space $\fX_{2,n}$ has exactly $2n+6$ holomorphic sections, consisting of $n$ tautological sections and their hyperelliptic conjugates and six Weierstrass  
sections. The hyperelliptic conjugates are of the form $J\circ s$, where $s$ is a tautological section and $J$ is the fiber-preserving involution whose restriction to each fiber of the Teichm\"uller curve is the hyperelliptic involution.\\
\indent Let $k$ be a field of characteristic zero. Assume that $2g-2+n >0$. The moduli stack of curves of type $(g, n)$ over $k$  is denoted by $\M_{g,n/k}$. Its hyperelliptic locus, denoted by $\cH_{g,n/k}$, is a closed smooth substack of $\M_{g,n/k}$. Let $\cC_{g,n/k}\to \M_{g,n/k}$ be the universal complete curve of type $(g,n)$ and  let $\pi:\cC_{\cH_{g,n/k}}\to \cH_{g,n/k}$ be the pullback of the universal curve to $\cH_{g,n/k}$.  We call it the universal complete hyperelliptic curve of type $(g,n)$, or simply, the universal hyperelliptic curve.   It admits the hyperelliptic involution $J$.    Denote the tautological sections of $\cC_{g,n/k}\to \M_{g,n/k}$ by $x_1, \ldots, x_n$, and the tautological sections of $\pi:\cC_{\cH_{g,n/k}}\to \cH_{g,n/k}$ induced by pulling back the sections $x_j$ are denoted by $s_1, \ldots, s_n$.  The universal hyperelliptic curve also admits the hyperelliptic conjugates $J\circ s_1, \ldots, J\circ s_n$ of the sections $s_j$.  In \cite{wat_sec}, it was shown that when $g\geq 3$ and $k$ is a field of characteristic zero with the image of the $\ell$-adic cyclotomic character $\chi_\ell: G_k\to \Z_\ell^\times$ is infinite, the sections of $\pi$ are exactly the tautological ones and their hyperelliptic conjugates.  Our first main result extends this fact to the universal complete hyperelliptic curve of type $(g,n)$ over $\C$. 
\begin{bigtheorem}\label{sections over C}
If $g\geq 3$ and $n\geq 0$, then the universal hyperelliptic curve $\pi: \cC_{\cH_{g,n/\C}}\to \cH_{g,n/\C}$ admits exactly $2n$ sections, consisting of $s_1, \ldots, s_n$ and $J\circ s_1, \ldots, J\circ s_n$. 
\end{bigtheorem}
A result \cite[Thm.~10.3 (a)]{EaKr} of Earle and Kra states that the Teichm\"uller curve over $\fX_{2, n}$ with $n\geq 1$ has exactly $n-1$ holomorphic sections disjoint from $s_1$. Together with Theorem \ref{sections over C}, we extend this result to the universal complete hyperelliptic curve $\pi: \cC_{\cH_{g,n/\C}}\to \cH_{g,n/\C}$.
\begin{bigtheorem}
If $g\geq 3$ and $n\geq 0$, then the sections of $\pi: \cC_{\cH_{g,n/\C}}\to \cH_{g,n/\C}$ that are disjoint from $s_1$ are exactly the tautological sections $s_2, \ldots, s_n$. 
\end{bigtheorem}
This result should be considered as a partial extension of Earle and Kra's work, since our method does not work for the universal hyperelliptic curve over the hyperelliptic locus of $\fX_{g,n}$.  While the approach of Earle and Kra is complex analytic by nature, our approach is more algebraic and topological.\\
\indent Our third result concerns the punctured case. The universal punctured curve of type $(g,n)$ is given by $\M_{g,n+1/k}\to \M_{g,n/k}$, which is the restriction of the universal complete curve to the complement of the tautological sections $x_j$. Pulling back  $\M_{g,n+1/k}\to \M_{g,n/k}$ to $\cH_{g,n/k}$, we obtain the universal punctured hyperelliptic curve $\pi^o: \cH_{g,n+1/k}\to \cH_{g,n/k}$ of type $(g,n)$. The fiber of $\pi^o$ over a geometric point $[C; \bar x_1, \ldots, \bar x_n]$ is the $n$-punctured curve, denoted by $C^o$, that is the complement of the $n$ marked points in $C$. 
 Associated to the curve  $\pi^o$, there is the homotopy exact sequence of algebraic fundamental groups, 
\begin{equation}\label{main punct seq for hyp}
1\to\pi_1^\alg(C^o)\to \pi_1^\alg(\cH_{g,n+1/\C})\to\pi_1^\alg(\cH_{g,n/\C})\to 1. 
\end{equation}
\begin{bigtheorem}\label{main seq does not split}
If $g\geq 3$ and $n\geq 0$, then the sequence (\ref{main punct seq for hyp}) does not split. 
\end{bigtheorem}
Let $S$ be a compact oriented surface of genus $g$ with $n$ marked points. 
Denote the hyperelliptic mapping class group of $S$ fixing the $n$ marked points pointwise by $\Delta_{g,n}$(defined in \S \ref{def of hyp map grps}). The punctured surface $S_{g,n}$ is the complement of the marked points in $S$ and its topological fundamental group is denoted by $\pi_1(S_{g,n})$. As an immediate consequence of Theorem \ref{main seq does not split}, we obtain the result that an analogue of the Birman exact sequence for the hyperelliptic mapping class groups does not split. 
\begin{bigcorollary}\label{hyp birman seq not split}
If $g\geq 3$ and $n\geq 0$, then the sequence
$$
1\to \pi_1(S_{g,n})\to \Delta_{g,n+1}\to \Delta_{g,n}\to 1
$$
does not split. 
\end{bigcorollary}
\indent This paper is a sequel to \cite{wat_sec}, where the main tool used is the weighted completion of profinite groups developed by Hain and Matsumoto in \cite{HaMa_wcomp}. In order to study the universal hyperelliptic curve over $\C$ or an algebraically closed field $k\subset \C$, we use the relative completion of the hyperelliptic mapping class groups in this paper. The Lie algebra of the relative completion admits a canonical mixed Hodge structure (MHS)  and hence a natural weight filtration with weights $\leq 0$. The key obstruction for the sections of the universal hyperelliptic curves is built in the Lie algebra structure of the 4-step graded Lie algebras constructed from the universal hyperelliptic curve via relative completion (defined in \S \ref{proofs}). 
%%%%%%%%%%%%%%%%%%%%%%%%%%%%%%%%%%%%%%%%%%%%%%%%%%%%%%%%%%%%%%%%%%%%%%%%%%%%%%%%%%%%%%%%%%%%%%%%%%%%%%%%%%%%%%%%
\section{Fundamental groups}
In this paper, three different types of fundamental groups are considered. For a connected topological space $X$ with a base point $x\in X$, denote the topological fundamental group of $X$ by $\pi_1(X, x)$. More generally, for a connected orbifold $\O$ with a base point $x$ in the underlying topological space $X_\O$,  denote the orbifold fundamental group of $\O$ by $\pi_1(\O, x)$ (see \cite{thurston} for the definition). Let $k$ be a field. For a geometrically connected scheme or more generally a connected DM (Deligne-Mumford) stack $\X_{/k}$ over $k$, denote the algebraic fundamental group of $\X_{/k}$ with a geometric point $\bar x$ by $\pi_1^\alg(\X_{/k}, \bar x)$ (see \cite{noo_alg_stacks} for the definition). Denote the profinite completion of a group $G$ by $\widehat G$. By the Riemann existence theorem extended for stacks \cite[Thm.~20.4, Cor.~20.5]{noo_top_stacks},  if $\X_{/\C}$ is a connected algebraic stack that is locally of finite type over $\C$, there is an isomorphism:
$$
\pi_1^\alg(\X_{/\C}) \cong \widehat{\pi_1(\X)},
$$
where $\X$ is the underlying orbifold of $\X_{/\C}$. For a field $k\subset \C$, let $\bar k$ be the algebraic closure of $k$ in $\C$. Then the base change from $\bar k$ to $\C$ gives an isomorphism 
$$
\pi_1^\alg(\X_{\bar k})\cong \pi_1^\alg(\X_\C).
$$
%%%%%%%%%%%%%%%%%%%%%%%%%%%%%%%%%%%%%%%%%%%%%%%%%%%%%%%%%%%%%%%%%%%%%%%%%%%%%%%%%%%%%%%%%%%%%%%%%%%%%%%%
\subsection{Surface groups and Pure braid groups}
Assume that $g \geq1$. Let $S$ be a compact oriented surface of genus $g$. Fix a base point $x\in S$ and let $\alpha_1, \beta_1, \ldots, \alpha_g, \beta_g$ form the standard generators for  $\pi_1(S, x)$. 
Denote the $n$-punctured surface obtained from $S$ by removing $n$ distinct points of $S$ by $S_{g,n}$. Fix a base point $y\in S_{g,n}$.  The fundamental group $\pi_1(S_{g,n}, y)$ has a presentation
$$
\pi_1(S_{g,n},y) \cong \langle \alpha_1, \beta_1,\ldots, \alpha_g, \beta_g, \gamma_1, \ldots, \gamma_n|
[\alpha_1, \beta_1][\alpha_2, \beta_2]\cdots[\alpha_g, \beta_g]\gamma_1\cdots\gamma_n=1\rangle,
$$
where  the $\gamma_i$ are the loops around the punctures and $[\alpha_j, \beta_j]=\alpha_j\beta_j\alpha_j^{-1}\beta_j^{-1}$. \\
\indent The $n$th ordered configuration space of $S$ is defined to be
$$
F^n(S) = S^n - \{(x_1, x_2, \ldots, x_n)\in S^n| x_i = x_j\,\,\text{ for }i\not= j\}.
$$
Fix a base point $z\in F^n(S)$. The {\it $n$-strand pure braid group} on $S$  is the topological fundamental group
$\pi_1(F^n(S), z)$.

%%%%%%%%%%%%%%%%%%%%%%%%%%%%%%%%%%%%%%%%%%%%%%%%%%%%%%%%%%%%%%%%%%%%%%%%%%%%%%%%%%%%%%%%%%%%%%%%%%%%%%%%%%%%%%%%
\section{Mapping class groups and hyperelliptic mapping class groups}\label{def of hyp map grps}
Let $S$ be a smooth compact oriented surface of genus $g$ and $P$ a subset of $S$ consisting of $n$ distinct points.  Assume that $2g -2 + n >0$. Define the mapping class group $\G_{g,n}$ to be the group of isotopy classes of  orientation-preserving diffeomorphisms of $S$ that fix $P$ pointwise:
$$
\G_{g,n} = \pi_0\Diff^+(S, P).
$$
When $n =0$, we denote $\G_{g,0}$ by $\G_g$. \\
\indent Fix a hyperelliptic involution $\sigma$ of $S$. It is an orientation-preserving diffeomorphism of $S$ of order $2$ with $2g+2$ fixed points. The hyperelliptic mapping class group $\D_g$ is defined as the group of isotopy classes of orientation-preserving diffeomorphisms that commute with $\sigma$: 
$$
\D_g = \pi_0(\text{centralizer of $\sigma$ in $\Diff^+S$}).
$$
By the result \cite{birman-hilden} of Birman and Hilden,  the natural homomorphism $\D_g\to \G_g$ is injective and its image is the centralizer of the isotopy class of $\sigma$ in $\G_g$. Therefore, $\D_g$ is identified with its image in $\G_g$ in this paper. Forgetting the $n$ marked points of $(S, P)$ yields a natural projection $\G_{g,n}\to\G_g$. Define the hyperelliptic mapping class group $\D_{g,n}$ of type $(g,n)$  to be the fiber product
$$
\D_{g,n} = \D_g\times_{\G_g}\G_{g,n}.
$$
We have the natural inclusion $\D_{g,n}\to \G_{g,n}$. 
%%%%%%%%%%%%%%%%%%%%%%%%%%%%%%%%%%%%%%%%%%%%%%%%%%%%%%%%%

%%%%%%%%%%%%%%%%%%%%%%%%%%%%%%%%%%%%%%%%%%%%%%%%%%%%%%%%%%%%%%%%%%%%%%%%%%%%%%%%%%%%%%%%%%%%%%%%%%%%%%%%
\subsection{Symplectic representation} \label{monodromy}
%Fix $x$ in $S$. Let $\pi_1(S, x)$ be the topological fundamental group of $S$ with the base point $x$. 
\textcolor{black}{For $A=\Z$ and $\Z/m\Z$ with an integer $m\geq 0$, denote $H_1(S, A)$  by $H_A$.} It is equipped with an algebraic intersection pairing $\langle~, ~\rangle: H_A^{\otimes 2}\to A$, which is a unimodular symplectic form. 
Denote the automorphism group of $H_A$ preserving $\langle~, ~\rangle$ by $\Sp(H_A)$. Fixing a symplectic basis for $H_A$ determines an isomorphism $\Sp(H_A)\cong \Sp_{2g}(A)$, the group of $2g\times 2g$ symplectic matrices with entries in $A$.  \\
\indent The action of $\G_{g,n}$ on $H_1(S, \Z)$ preserves the intersection pairing, and hence there is a natural representation
$$
\rho: \G_{g,n}\to \Sp(H_\Z).
$$
It is well known that $\rho$ is surjective for $g \geq 1$ (see  \cite[Thm.~6.4]{FaMa}). Restricting to $\D_{g,n}$, we obtain the natural representation 
$$
\rho^\hyp: \D_{g,n}\to \Sp(H_\Z).
$$
Note that the image of $\rho^\hyp$ does not depend on $n$. Denote the image of $\rho^\hyp$ by $G_g$. 
%The action of $\D_{g,n}$ on $H_\Z$ factors through $\D_g$, and so the image of $\rho^\hyp_{g,n}$ also contains $\Sp(H_\Z)[2]$. 

%%%%%%%%%%%%%%%%%%%%%%%%%%%%%%%%%%%%%%%%%%%%%%%%%%%%%%%%
\subsection{Level subgroups}
For each integer $m \geq0$, the congruence subgroup $\Sp(H_\Z)[m]$ of $\Sp(H_\Z)$ of level $m$ is defined to be the kernel of the reduction mod $m$ map:
$$
\Sp(H_\Z)[m] = \ker(\Sp(H_\Z) \to \Sp(H_{\Z/m\Z})).
$$
For each integer $m \geq 0$, define the level $m$ subgroup of $\G_{g,n}$ to be the kernel of the reduction of $\rho$ mod $m$:
$$
\G_{g,n}[m] = \ker\left( \G_{g,n}\overset{\rho}\to \Sp(H_\Z)\to \Sp(H_{\Z/m\Z})\right).
$$
The Torelli group $T_{g,n}$ is the level $0$ subgroup $\G_{g,n}[0]$. The level $m$ subgroup $\Delta_{g,n}[m]$ is the intersection of $\Delta_{g,n}$ with $\G_{g,n}[m]$:
$$
\Delta_{g,n}[m]=\Delta_{g,n}\cap \G_{g,n}[m].
$$
In particular, the level $0$ subgroup $\Delta_g[0]$ is the {\it hyperelliptic Torelli group} denoted by $T\Delta_g$. 
By a result of A'Campo \cite{acampo}, the image $G_g$ contains the principal congruence subgroup $\Sp(H_\Z)[2]$. For $m\geq 1$, denote the image of $\Delta_g[m]$ in $\Sp(H_\Z)$ by $G_g[m]$. In particular, $G_g[1]=G_g$.\\
\indent Let $W$ be the set of the $2g+2$ points fixed by $\sigma$. Fixing a labeling $W =\{1, 2, \ldots, 2g+2\}$, we may identify $\Aut W$ with the symmetric group $\mathbb{S}_{2g+2}$. The hyperelliptic mapping class group $\Delta_g$ acts on $W$ and the natural homomorphism $\rho_W:\Delta_g\to \Aut W$ is surjective. There is a faithful representation $\Aut W\hookrightarrow \Sp(H_{\Z/2\Z})$ (Cf. \cite[\S 6]{Mumford}), and there is a commutative diagram:
$$
\xymatrix@R=1em@C=2em{
\Delta_g\ar[r]^{\rho^\hyp}\ar[d]_{\rho_W}&\Sp(H_\Z)\ar[d]\\
\Aut W\ar[r]&\Sp(H_{\Z/2\Z}).
}
$$
It follows that the image of $G_g$ in $\Sp(H_{\Z/2\Z})$ is isomorphic to $\mathbb{S}_{2g+2}$.

%%%%%%%%%%%%%%%%%%%%%%%%%%%%%%%%%%%%%%%%%%%%%%%%%%%%%%%%
\subsection{Birman exact sequences} Recall that the punctured surface $S - P$ is denoted by $S_{g,n}$.   
If $2g -2+n > 0$, then there is the Birman exact sequence (see \cite[Thm.~4.6]{FaMa})
\begin{equation}\label{full birman seq}
1 \to \pi_1(S_{g,n}) \to \G_{g,n+1}\to \G_{g,n}\to 1.
\end{equation}
It is known that the Birman exact sequence does not split for $g=2$ and $n = 0$ (see \cite[Cor. 5.11]{FaMa}).  The  case when $g \geq 2$ and $n\geq 2$ follows from a result \cite[Cor.~4]{GG} of Gonçalves and Guaschi. Applying profinite completion to the sequence (\ref{full birman seq}), we obtain the profinite Birman exact sequence, and we have
\begin{theorem}[{\cite[Thm.~1, Cor.~2]{wat_rk}}]
If $g\geq 4$ and $n\geq 0$, the exact sequence
\begin{equation*}\label{pro full birman seq}
1 \to \widehat{\pi_1(S_{g,n})}\to \widehat {\G_{g,n+1}}\to \widehat{\G_{g,n}}\to 1
\end{equation*}
does not split. Consequently, if $g\geq 4$ and $n\geq 0$, then the sequence (\ref{full birman seq}) does not split. 
\end{theorem}
By pulling back the sequence (\ref{full birman seq}) along the natural inclusion $\D_{g,n}\to \G_{g,n}$, we obtain the Birman exact sequence for hyperelliptic mapping class groups:
\begin{equation}\label{hyp birman seq}
1 \to \pi_1(S_{g,n})\to \D_{g,n+1}\to \D_{g,n}\to 1.
\end{equation}
Taking the profinite completion of this sequence, we also have the profinite Birman exact sequence for hyperelliptic mapping class groups
\begin{equation}\label{pro hyp birman seq}
1 \to \widehat{\pi_1(S_{g,n})}\to \widehat {\D_{g,n+1}}\to \widehat{\D_{g,n}}\to 1.
\end{equation}
%%%%%%%%%%%%%%%%%%%%%%%%%%%%%%%%%%%%%%%%%%%%%%%%%%%%%%%%%%%%%%%%%%%%%%%%%%%%%%%%%%%%%%%%%%%%%%%%%%%%%%%%%%%%%%%%
\section{Moduli stack of hyperelliptic curves}   

%%%%%%%%%%%%%%%%%%%%%%%%%%%%%%%%%%%%%%%%%%%%%%%%%%%%%%%%
\subsection{The moduli stack of curves of type $(g,n)$}
Let $k$ be a field. The moduli stacks $\M_{g,n/k}$ and $\cH_{g,n/k}$ are the quotient of a smooth variety by a finite group. Since this group action is not free, in order to consider the universal curves and local systems, we are lead to consider these moduli spaces as stacks and their underlying spaces as orbifolds. \\
\indent   Let $\X_{/k}$ be a stack over $k$. A curve $f: C \to \X_{/k}$ of type $(g,n)$  over $k$ is a proper 
smooth morphism whose geometric fiber is an irreducible one-dimensional variety of arithmetic genus $g$, with $n$ disjoint sections $s_1, s_2,\ldots, s_n: \X_{/k}\to C$.  A punctured curve $f^o:C^o\to \X_{/k}$ of type $(g,n)$ over $\X_{/k}$ is the restriction of $f$ to the complement of the images of the $n$ sections in $C$. Denote by $\M_{g,n/k}$ the moduli stack of curves of type $(g,n)$ over $k$. The stack $\M_{g,n/k}$ is a smooth DM-stack over $k$. For a field $k' \supset k$, denote the base change  $\M_{g,n/k}\times_{\Spec k} \Spec k'$ by $\M_{g,n/k'}$.  The stack $\M_{g,n/k}$ admits the universal complete curve $\cC_{g,n/k}\to \M_{g,n/k}$ of type $(g,n)$, which is equipped with $n$ tautological sections $x_1, \ldots, x_n$. The restriction of the universal complete curve to the complement of the images of the tautological sections is the universal punctured curve $\M_{g, n+1/k}\to \M_{g,n/k}$ of type $(g,n)$ over $k$. \\
\indent For each integer $m \geq 1$, a Jacobi structure of level $m$ on a curve $f:C\to \X_{/k}$ is a symplectic isomorphism of sheaves over $\X_{/k}$
$$
\phi: R^1f_\ast(\Z/m\Z)\to (\Z/m\Z)^{2g},
$$
where the sheaf $R^1f_\ast(Z/m\Z)$ is provided with a symplectic structure with the cup product and $(\Z/m\Z)^{2g}$ is quipped with the standard symplectic structure. Let $k$ be a field containing all $m$th roots of unity $\mu_m(\bar k)$. Denote the DM stack classifying curves of type $(g,n)$ over $k$ with a Jacobi structure of level $m$ by $\M_{g,n/k}[m]$. It is a finite etale cover of $\M_{g,n/k}$. It follows from \cite[Thm.~1.10]{MFK} that when $m\geq 3$, it is a smooth quasi-projective variety over $k$. When $m=1$, the stack $\M_{g,n/k}[1]$ is denoted by $\M_{g,n/k}$. 
%%%%%%%%%%%%%%%%%%%%%%%%%%%%%%%%%%%%%%%%%%%%%%%%%%%%%%%%
\subsection{The universal hyperelliptic curves} 
Suppose that $g\geq 2$. Let $k$ be a field of characteristic zero. 
Denote the moduli stack of hyperelliptic curves of genus $g$ over $k$ by $\cH_{g/k}$. The stack $\cH_{g/k}$ is a closed substack of $\M_{g/k}$, which associates to every $k$-scheme $T$ the set
$$
\cH_{g/k}(T) =\{\text{curves } C\to T \text{ of hyperelliptic curves of genus } g\}/\cong.
$$
Arsie and Vistori proved in \cite{ArVi} that it is an irreducible smooth DM stack of finite type over $k$ of dimension $2g-1$. Assume that $m$ is a nonnegative integer and $k$ contains all $m$th roots of unity $\mu_m(\bar k)$.  The fiber product 
$$
\cH_{g/k}\times_{\M_{g/k}}\M_{g/k}[m]
$$
is a finite disjoint union of mutually isomorphic connected components.  We discuss this number of components in \S \ref{orb app}. Define $\cH_{g/k}[m]$ as a particular component of the fiber product.   When $m_1$ divides $m_2$, we always assume that $\cH_{g/k}[m_2]$ is a finite etale cover of $\cH_{g/k}[m_1]$. Define $\cH_{g, n/k}[m]$ as the fibre product
$$
\cH_{g/k}[m]\times_{\M_{g/k}}\M_{g,n/k}.
$$
When $m=1$, denote $\cH_{g,n/k}[1]$ by $\cH_{g,n/k}$. When $m\geq 3$, it is a smooth subvariety of $\M_{g,n/k}[m]$.
Pulling back the universal complete curve $\cC_{g,n/k}\to \M_{g,n/k}$ of type $(g,n)$ to $\cH_{g,n/k}[m]$, we have the universal complete hyperelliptic curve of type $(g,n)$ with a level $m$ structure
$$
\pi:\cC_{\cH_{g,n/k}}[m]\to \cH_{g,n/k}[m].
$$
Similarly, pulling back the universal punctured curve $\M_{g,n+1/k}\to \M_{g,n/k}$ of type $(g,n)$ to $\cH_{g,n/k}[m]$, we have the universal punctured hyperelliptic curve of type $(g,n)$ with a level $m$ structure
$$
\pi^o:\cH_{g,n+1/k}[m]\to \cH_{g,n/k}[m].
$$
The $n$ tautological sections $x_1, \ldots, x_n$ of $\cC_{g,n/k}\to \M_{g,n/k}$ pull back to give $n$ disjoint sections of $\pi$, which we denote by $s_1, s_2, \ldots, s_n$. The universal hyperelliptic curve $\pi$ admits a fiber-preserving automorphism of order $2$, denoted by $J$, whose restriction to each fiber is the hyperelliptic involution.   By composing with $J$, we obtain sections $J\circ s_1, \ldots, J\circ s_n$, which we call the hyperelliptic conjugates of the sections $s_j$. The universal punctured hyperelliptic curve $\pi^o$ is the restriction of $\pi$ to the complement of the images of the sections $s_j$ in $\cC_{\cH_{g,n/k}}[m]$. 
%%%%%%%%%%%%%%%%%%%%%%%%%%%%%%%%%%%%%%%%%%%%%%%%%%%%%%%%
\subsection{The orbifold approach}\label{orb app}
Here we mean a compact Riemann surface by a complex curve. Suppose that $2g -2 +n > 0$. Recall that the pair $(S, P)$ is an $n$-pointed smooth compact oriented surface $S$ of genus $g$. Denote the Teichm\"uller space of $(S,P)$ by $\fX_{g,n}$. When $n=0$, denote $\fX_{g,0}$ by $\fX_g$.  It is a set of isotopy classes of orientation-preserving diffeomorphisms  $h: S\to C$ of $S$  to a complex curve $C$. In fact, it is a contractible complex analytic manifold of dimension $3g -3 +n$. The mapping class group $\G_{g,n}$ acts on $\fX_{g,n}$ as a group of biholomorphisms via its action on the markings:
$$
\lambda: h\mapsto h\circ \lambda^{-1},\,\, \lambda \in \G_{g,n},\,\, f \in \fX_{g,n}.
$$
It is well known that this action is properly discontinuous (see \cite[Thm.~12.2]{FaMa}) and virtually free.\\
\indent As an orbifold, for each integer $m\geq 0$, the moduli space of $n$-pointed smooth projective curves of genus $g$ with a level $m$ structure is defined to be the orbifold quotient of $\fX_{g,n}$ by $\G_{g,n}[m]$:
$$
\M_{g,n}[m]= (\G_{g,n}[m]\backslash\fX_{g,n})^\orb.
$$
The orbifold $\M_{g,n}[m]$ is the underlying orbifold of $\M_{g,n/\C}[m]$.\\
\indent  Recall that $\sigma:S\to S$ is a fixed hyperelliptic involution of $S$. Let $\mathfrak{Y}_g$ be the set of points of $\fX_g$ fixed by $\sigma$: 
$$
\mathfrak{Y}_g = \fX_g^\sigma.
$$
It consists of markings $[h:S\to C]$ such that $h\sigma h^{-1} \in \Aut C$. By a result of Earle \cite{Ea}, it is biholomorphic to the Teichm\"uller space $\fX_{0, 2g+2}$, and therefore connected and contractible of dimension $2g-1$.  Observe that the stabilizer of $\mathfrak{Y}_g$ is the hyperelliptic mapping class group $\Delta_g$. Since there is a unique conjugacy class of hyperelliptic involutions in $\G_g$, it follows that the hyperelliptic locus $\fX_g^\hyp$ in $\fX_g$ is
$$
\fX_g^\hyp = \bigcup_{\lambda\in \G_g/\Delta_g}\lambda(\mathfrak{Y}_g) =  \bigcup_{\lambda\in \G_g/\Delta_g}\fX_g^{\lambda\sigma\lambda^{-1}}.
$$
 As an orbifold,  for an integer $m \geq 0$, the moduli space of smooth projective hyperelliptic curves of genus $g$ with a level $m$ structure is the orbifold quotient of $\mathfrak{Y}_g$ by $\Delta_g[m]$:
$$
\cH_g[m] = (\Delta_g[m]\backslash\mathfrak{Y}_g)^\orb.
$$
The orbifold $\cH_g[m]$ is the underlying orbifold of $\cH_{g/\C}[m]$. 
Since each hyperelliptic curve $C$ has a unique hyperelliptic involution, we have 
\begin{proposition}
The hyperelliptic locus $\fX_g^\hyp$ of Teichm\"uller space is the disjoint
union of the translates of $\mathfrak{Y}_g$:
$$
\fX_g^\hyp = \coprod_{\lambda \in \G_g/\D_g} \lambda(\mathfrak{Y}_g). \qed
$$
\end{proposition}

When $m\geq 0$ is even, $\Sp(H_\Z)[m] \subseteq G_g$. In this case there is
an isomorphism
$$
\Sp(H_\Z)/\big(G_g\cdot\Sp(H_\Z)[m]\big) \cong \Sp(H_{\Z/2\Z})/\mathbb{S}_{2g+2}.
$$
Since $\Sp(H_\Z)[m]$ is generated by symplectic transvections, %\textcolor{red}{(ref?)}
the
reduction mod 2 mapping $\Sp(H_\Z)[m] \to \Sp(H_{\Z/2\Z})$ is surjective whenever
$m$ is odd.

\begin{corollary}
The hyperelliptic locus of $\M_g[m]$ is
\begin{align*}
\M_g^\hyp[m] &=
\coprod_{\phi \in \Sp(H_{\Z})/(G_g\cdot\Sp(H_{\Z}[m]))} \phi(\cH_g[m])\cr
&=
\begin{cases}
\cH_g[m] & m \text{ odd},\cr
\coprod_{\phi \in \Sp(H_{\Z/2\Z})/\mathbb{S}_{2g+2}} \phi(\cH_g[m]) & m\ge 0 \text{ even }.
\end{cases}
\end{align*}
Consequently, for each $m\ge 0$, the number of components of the hyperelliptic
locus in $\M_g[2m]$ equals
$$
|\Sp(H_{\Z/2\Z})|/|\mathbb{S}_{2g+2}| =
\frac{2^{g^2}\prod_{j=1}^g (2^{2j}-1)}{(2g+2)!}. \qed
$$
\end{corollary}
 The Torelli space $\T_g$ is the quotient of $\fX_g$ by the Torelli group $T_g$: $\T_g = T_g\backslash \fX_g$. The hyperelliptic locus $\T^\hyp_g$ in $\T_g$ is the image of $\fX^\hyp_g$ in $\T_g$. It is a finite disjoint union of the translates of $\cH_g[0] = T\Delta_g\backslash \mathfrak{Y}_g$.  The action of $\Delta_g[m]$ on $\cH_g[0]$ factors through $G_g[m]$, and the quotient $G_g[m]\backslash \cH_g[0]$ is the orbifold $\cH_g[m]$.

%%%%%%%%%%%%%%%%%%%%%%%%%%%%%%%%%%%%%%%%%%%%%%%%%%%%%%%%
\subsection{Curves of compact type}
Denote by $\overline{\M}_{g,n/\C}$ the Deligne-Mumford compactification (see \cite{DM}) of $\M_{g,n/\C}$. The stack $\overline{\M}_{g,n/\C}$ classifies the isomorphism classes of stable $n$-pointed curves of genus $g$. It is proper and smooth over $\C$ and the boundary $\overline{\M}_{g,n/\C} - \M_{g,n/\C}$ is a divisor with normal crossings. Denote by $\delta_0$ the component of the boundary divisor whose generic point is an irreducible $n$-pointed curve of genus $g$ with one node. The dual graph of the curves of this type is not a tree. A projective nodal complex curve $C$ is said to be of {\it compact type} if $\Pic^0C$ is an abelian variety or equivalently its dual graph is a tree. If $C$ is of compact type, then the first homology of $C$ is the sum of the first homology groups of its components:
$$
H_1(C,\Z) \cong \bigoplus H_1(C_i,\Z).
$$
The sum of the intersection forms of its components yields a unimodular symplectic form on $H_1(C;\Z)$. \\
%A level $m$ structure on a curve $C$ of compact type is a symplectic basis of $H^1(C;\Z/m\Z)$. \\
\indent Denote the complement of $\delta_0$ in $\overline{\M}_{g,n/\C}$ by $\M^c_{g,n/\C}$. The stack $\M^c_{g,n/\C}$ classifies the stable $n$-pointed curves of compact type of genus $g$. Denote the moduli stack of principally polarized abelian varieties of dimension $g$  over $\C$ by $\A_{g/\C}$.   The Torelli map $\M_{g,n/\C}\to \A_{g/\C}$ extends to give a map $\M^c_{g,n/\C}\to \A_{g/\C}$. 
%Let $\mathfrak{f}:\mathfrak{\X_{g/\C}}\to \A_{g/\C}$ be the universal family of principally polarized abelian varieties of dimension $g$ over $\C$. 
For each $m\geq 1$, 
%a level $m$ structure on $\mathfrak{f}:\mathfrak{\X_{g/\C}}\to \A_{g/\C}$ is a symplectic isomorphism $\phi:R^1\mathfrak{f}_\ast(\Z/m\Z)\to (\Z/m\Z)^{2g}$. 
denote by $\A_{g/\C}[m]$ the moduli stack of principally polarized abelian varieties of dimension $g$ with a level $m$ structure. It is a finite etale cover of $\A_{g/\C}$ and a smooth quasi-projective complex variety for $m \geq 3$. The pullback of $\A_{g/\C}[m]$ to $\M^c_{g,n/\C}$ along the Torelli map is the moduli stack $\M^c_{g,n/\C}[m]$ of curves of compact type of type $(g,n)$ with a level $m$ structure. The closure of $\cH_{g,n/\C}[m]$ in $\M^c_{g,n/\C}[m]$ is denoted by $\cH^c_{g,n/\C}[m]$. \\
\indent  By the deformation theory of stable curves, $\T_g$ can be enlarged to a complex manifold, denoted by $\T^c_g$. It is the Torelli space of curves of compact type of genus $g$ with a homology framing.  The action of $\Sp(H_\Z)[m]$ on $\T_g$ extends to $\T^c_g$, and the quotient $\Sp(H_\Z)[m]\backslash \T^c_g$ is the underlying orbifold $\M^{c}_{g}[m]$ of $\M^c_{g/\C}[m]$.  \\
\indent The closure of the hyperelliptic locus $\T^\hyp_g$ in $\T^c_g$ is denoted by $\T^{hyp,c}_g$. It is a complex manifold consisting of finitely many mutually isomorphic smooth components. 
Brendle, Margalit, and Putman proved in \cite{BMP} that each connected component of $\T^{\hyp,c}_g$ is simply-connected. The component of $\T^{\hyp,c}_g$ corresponding to the closure of $\cH_g[0]$ is denoted by $\cH^{c}_g[0]$. Moreover, for $m\geq 1$, denote the closure of $\cH_g[m]$ in $\M^{c}_g[m]$ by $\cH^{c}_g[m]$. The action of $G_g[m]$ on $\cH_g[0]$ extends to $\cH^{c}_g[0]$, and the quotient $G_g[m]\backslash \cH^{c}_g[0]$ is the underlying orbifold $\cH^{c}_g[m]$ of  the stack $\cH^c_{g/\C}[m]$. 
When $m\geq 3$, $\Sp(H_\Z)[m]$ is torsion-free  and so is $G_g[m]$. Therefore $\M^c_{g}[m]$ and $\cH^c_{g}[m]$ and  are complex manifolds that are smooth quasi-projective complex varieties. 
%When $m\geq 3$, $\cH^c_{g/\C}[m]$ is a smooth variety over $\C$.  %\textcolor{red}{Remark on the case $m\geq 3$?}

%%%%%%%%%%%%%%%%%%%%%%%%%%%%%%%%%%%%%%%%%%%%%%%%%%%%%%%%
\subsection{Key homotopy exact sequences}
Assume that $g \geq 2$, $n\geq 0$, and $m\geq 1$.  Here we introduce the homotopy exact sequences of the orbifold fundamental groups associated to the universal hyperelliptic curves 
$\pi:\cC_{\cH_{g,n}}[m]\to \cH_{g,n}[m]$ and $\pi^o:\cH_{g, n+1}[m]\to \cH_{g,n}[m]$. For the algebraic fundamental groups, we apply profinite completion to obtain the corresponding exact sequences. \\
\indent For $n\geq 1$, denote the $n$th power of  $\cC_{\cH_g}[m]$ by $\cC^n_{\cH_{g} }[m]$. The moduli space $\cH_{g,n+1}[m]$ is the complement of the images of the tautological sections $s_1, \ldots, s_n$ in $\cC_{\cH_{g,n}}[m]$, and on the other hand $\cC_{\cH_{g,n}}[m]$ is open in $\cC^{n+1}_{\cH_{g} }[m]$.  For $n\geq 1$, let $\etabar_{n+1}=[C;\bar x_0,\bar x_1,\ldots,\bar x_n]$ be a point of $\cH_{g,n+1}[m]$, $\etabar_n =[C;\bar x_1,\ldots,\bar x_n]$ the image of $\etabar_{n+1}$ in $\cH_{g,n}[m]$ via $\pi^o$, and $\etabar_0=[C]$ the image of $\etabar_n$ via the projection $\cH_{g,n}[m]\to \cH_{g}[m]$. These points are also considered as geometric points in the corresponding algebraic stacks. We consider $\etabar_{n+1}$ as a point of $\cC_{\cH_{g,n}}[m]$ via the open immersion $\cH_{g,n+1}[m]\hookrightarrow \cC_{\cH_{g,n}}[m]$. Let $C$ and $C^o$ be the fibers  over $\etabar_n$ of $\pi$ and $\pi^o$, respectively. Note that the point $\bar x_0$ is a geometric point of $C^o$.  Let $\pi': \cC_{\cH_{g,n}}'[m]\to \cH_{g,n}[m]$ be the restriction of $\pi$ to the complement of the image of the section $s_{1}$ in $\cC_{\cH_{g,n}}[m]$. The fiber of $\pi'$ over $\etabar_n$ is the $1$-punctured curve, denoted by $C'$,  obtained from $C$ by removing $\bar x_1$. \\
\indent Denote the $n$-th ordered configuration space of $C$ by $F^n(C)$. The fibers over $\etabar_0$ of $\cH_{g,n}[m]\to \cH_{g}[m]$ and $\cH_{g, n+1}[m]\to \cH_{g}[m]$ are the configuration spaces $F^n(C)$ and $F^{n+1}(C)$ on $C$, respectively.  The points $\bar {\bf x}_n = (\bar x_1, \ldots, \bar x_n)$ and $\bar {\bf x}_{n+1} =(\bar x_0, \bar x_1, \ldots, \bar x_{n})$ are points of $F^n(C)$ and $F^{n+1}(C)$, respectively. The universal punctured hyperelliptic curve $\pi^o$ induces the projection $F^{n+1}(C)\to F^n(C)$ given by $(u_0,u_1,\ldots, u_n)\mapsto (u_1, \ldots, u_n)$, whose fiber over $\bar {\bf x}_n$ is $C^o$. Then there are exact sequences of fundamental groups:
\begin{equation}\label{punct homo seq}
1\to \pi_1(C^o, \bar x_0)\to \pi_1(\cH_{g, n+1}[m], \etabar_{n+1})\to \pi_1(\cH_{g,n}[m], \etabar_n)\to 1,
\end{equation}
\begin{equation}\label{1 punct homo seq}
1\to \pi_1(C', \bar x_0)\to \pi_1(\cC_{\cH_{g,n}}'[m], \etabar_{n+1})\to \pi_1(\cH_{g,n}[m], \etabar_n)\to1.
\end{equation}
\begin{equation}\label{complete homo seq}
1\to \pi_1(C, \bar x_0)\to \pi_1(\cC_{\cH_{g, n}}[m], \etabar_{n+1})\to \pi_1(\cH_{g, n}[m], \etabar_n)\to 1,
\end{equation}
\begin{equation}\label{n+1th power over nth seq }
1\to \pi_1(C, \bar x_0)\to \pi_1(\cC^{n+1}_{\cH_{g}}[m], \etabar_{n+1})\to \pi_1(\cC^{n}_{\cH_{g}}[m], \etabar_{n})\to 1,
\end{equation}
\begin{equation}\label{univ seq with conf fiber}
1\to \pi_1(F^n(C), \bar {\bf x}_n)\to \pi_1(\cH_{g,n}[m], \etabar_n)\to \pi_1(\cH_{g}[m], \etabar_0)\to 1,
\end{equation}
\begin{equation}\label{nth power seq}
1\to \prod_{j=1}^n\pi_1(C, \bar x_j)\to \pi_1(\cC^n_{\cH_{g}}[m], \etabar_n)\to \pi_1(\cH_{g}[m], \etabar_0)\to 1,
\end{equation}
and
\begin{equation}
1\to \pi_1(C^o, \bar x_0)\to \pi_1(F^{n+1}(C), \bar {\bf x}_{n+1})\to \pi_1(F^n(C), \bar {\bf x}_n)\to 1.
\end{equation}

These exact sequences fit in the following commutative diagrams:
\begin{equation}\label{comm diag fund grp}
\xymatrix@R=1em@C=2em{        
1\ar[r]&\pi_1(C^o, \bar x_0)\ar[r]\ar@{=}[d]&\pi_1(F^{n+1}(C), \bar {\bf x}_{n+1})\ar[r]\ar@{^{(}->}[d]  &\pi_1(F^n(C), \bar {\bf x}_n)\ar[r]\ar@{^{(}->}[d]        &1\\
1\ar[r]&\pi_1(C^o, \bar x_0)\ar[r]  \ar[d]      &\pi_1(\cH_{g, n+1}[m], \etabar_{n+1})\ar[r]\ar[d] & \pi_1(\cH_{g, n}[m], \etabar_n)\ar[r]\ar@{=}[d]  &1\\
1\ar[r]&\pi_1(C', \bar x_0)\ar[r]  \ar[d]      &\pi_1(\cC_{\cH_{g, n+1}}'[m], \etabar_{n+1})\ar[r]\ar[d] & \pi_1(\cH_{g, n}[m], \etabar_n)\ar[r]\ar@{=}[d]  &1\\
1\ar[r]&\pi_1(C, \bar x_0)\ar[r] \ar@{=}[d] &\pi_1(\cC_{\cH_{g,n}}[m], \etabar_{n+1})\ar[r]\ar[d]    & \pi_1(\cH_{g, n}[m], \etabar_n)\ar[r]\ar[d]&1\\
1\ar[r]&\pi_1(C, \bar x_0)\ar[r]&\pi_1(\cC^{n+1}_{\cH_{g}}[m], \etabar_{n+1})\ar[r]    &\pi_1(\cC^{n}_{\cH_{g}}[m], \etabar_{n})\ar[r]&1
}
\end{equation}
\begin{equation}\label{comm diag fund grp with the nth power}
\xymatrix@R=1em@C=2em{        
1\ar[r]&\pi_1(F^n(C), \bar {\bf x}_n)\ar[r]\ar[d]&\pi_1(\cH_{g, n}[m], \etabar_n)\ar[r]\ar[d]  &\pi_1(\cH_{g}[m], \etabar_0)\ar[r]\ar@{=}[d]        &1\\
1\ar[r]&\prod_{j=1}^n\pi_1(C, \bar x_j)\ar[r]  \ar@{=}[d]      & \pi_1(\cC^n_{\cH_{g}}[m], \etabar_n)\ar[r]\ar[d] & \pi_1(\cH_{g}[m], \etabar_0)\ar[r]\ar[d]^{\mathrm{diag} } &1\\
1\ar[r]&\prod_{j=1}^n\pi_1(C, \bar x_j)\ar[r]  &\prod_{j=1}^n\pi_1(\cH_{g, 1}[m], \bar x_j)\ar[r]    & \pi_1(\cH_{g}[m], \etabar_0)^n\ar[r]&1
 }
\end{equation}

%\textcolor{red}{Introduce the exact sequence for $\cC^n_{\cH_{g/k}}\to \cH_{g/k}$}
\begin{remark}
Similar commutative diagrams of algebraic fundamental groups exist and are obtained by taking the profinite completion of the diagrams (\ref{comm diag fund grp}) and (\ref{comm diag fund grp with the nth power}). 
\end{remark}

%%%%%%%%%%%%%%%%%%%%%%%%%%%%%%%%%%%%%%%%%%%%%%%%%%%%%%%%%%%%%%%%%%%%%%%%%%%%%%%%%%%%%%%%%%%%%%%%%%%%%%%%%%%%%%%%
\section{Minimal Presentations of Lie algebras $\Gr^W_\bullet \p$ and $\Gr^W_\bullet\p_{g,n}$} 
Let $\cP$,  $\cP^o$, and $\cP_{g,n}$ be the unipotent completions of $\pi_1(C, \bar x_0)$,  $\pi_1(C^o, \bar x_0)$, and $\pi_1(F^n(C), \bar{\bf x}_n)$ over $\Q$. Denote the Lie algebras of $\cP$,  $\cP^o$, and $\cP_{g,n}$ by $\p$,  $\p^o$, and $\p_{g,n}$, respectively. The Lie algebras $\p$ and $\p^o$
%of the unipotent completions $\cP$, $\cP'$, $\cP^o$, and $\cP_{g,n}$, respectively, 
admit natural weight filtrations as MHSs constructed by Morgan in \cite{morgan} and  $\p_{g,n}$ by Hain in \cite{hain_deRham}. Denote the weight filtrations on $\p$, $\p^o$, and $\p_{g,n}$ by $W_\bullet\p$, $W_\bullet\p^o$,  and $W_\bullet\p_{g,n}$, respectively. In this section, we will review presentations of the associated graded Lie algebras $\Gr^W_\bullet \p$ and $\Gr^W_\bullet\p_{g,n}$. For a vector space $V$, denote the free Lie algebra generated by $V$ by $\L(V)$. It is a graded Lie algebra:
$$
\L(V) = \bigoplus_{n\geq 1}\L_n(V),
$$
where $\L_n(V)$ is the component with bracket length $n$. 
Set $H = H_1(C,\Q)$. The group $H$ is equipped with the algebraic intersection pairing $\langle~,~\rangle: H\otimes H\to \Q$, which is a unimodular symplectic form. 

%%%%%%%%%%%%%%%%%%%%%%%%%%%%%%%%%%%%%%%%%%%%%%%%%%%%%%%%%%%%%
\subsection{A presentation of $\Gr^W_\bullet\p$} It follows from the fact that the $H_1(\p) =H_1(C)$ is pure of weight $-1$  that the natural weight filtration on $\p$ agrees with its lower central series (see \cite[Lemma 4.7]{hain_infini}), and hence it is defined over $\Q$. Let $a_1, b_1, \ldots, a_g, b_g$ be a symplectic basis for $H$. Let $\Theta =\sum_{l=1}^ga_l\wedge b_l$ in $\Lambda^2H$. From the work of Labute \cite{labute}, we have the following presentation of $\Gr^W_\bullet\p$. 
\begin{proposition}
If $g \geq 1$, there is a natural $\Sp(H)$-equivariant Lie algebra isomorphism
$$
\Gr^W_\bullet \p \cong \L(H)/ (\Theta).
$$
\end{proposition}
Denote the $\Sp(H)$-module $\Lambda^2H/(\Theta)$ by $\Lambda^2_0 H$. It is isomorphic to the irreducible $\Sp(H)$-representation $V_{[1^2]}$ corresponding to the partition $2 = 1+1$. 
\begin{corollary}\label{surface weight -1 -2}
If $g\geq 2$, then there are $\Sp(H)$-equivariant isomorphisms
$$
\Gr^W_r\p \cong \begin{cases}
			H & r =-1\\
			\Lambda^2_0H & r =-2	
			\end{cases}
$$
\end{corollary}
Recall that $C^o$ is the $n$-punctured curve obtained from $C$ by removing $n$ distinct points $\bar x_1, \ldots, \bar x_n$. For $n\geq 1$, the natural projection $\pi_1(C^o, \bar x_0)\to \pi_1(C, \bar x_0)$ induced by the open immersion $C^o\hookrightarrow C$ yields the surjection $\p^o\to \p$. This Lie algebra surjection is a morphism of MHS, and hence, there is the $\Sp(H)$-equivariant graded Lie algebra surjection $\Gr^W_\bullet \p^o\to \Gr^W_\bullet\p$ whose kernel is generated by $\Q(1)^{\oplus n}$ in weight $-2$. 
\begin{corollary}\label{open surface weight -1 -2}
If $g \geq 2$ and $n\geq 1$, then there are $\Sp(H)$-equivariant isomorphisms
$$
\Gr^W_r\p^o \cong \begin{cases}
			H & r =-1\\
			\Lambda^2_0H \oplus \Q(1)^{\oplus n} & r =-2	
			\end{cases}
$$

\end{corollary}
%%%%%%%%%%%%%%%%%%%%%%%%%%%%%%%%%%%%%%%%%%%%%%%%%%%%%%%%%%%%%
\subsection{A presentation of $\Gr^W_\bullet\p_{g,n}$} The Lie algebra structure of $\Gr^W_\bullet\p_{g,n}$ plays a key role in the study of the sections of the universal punctured hyperelliptic curve $\pi^o: \cH_{g, n+1}\to \cH_{g,n}$. Together with Morgan's Theorem in \cite{morgan}, the following result by Hain shows that $\Gr^W_{\bullet}\p_{g,n}$ has a quadratic presentation. 
\begin{proposition}[{\cite[Prop.~2.1\& Prop.~12.5]{hain_infini}}]\label{ab of pgn}
For each $[C]$ in $\M_g$ and for all $g\geq 1$ and $n\geq 1$, there are isomorphisms of MHSs
$$
H_1(\p_{g,n})\cong H_1(\pi_1(F^n(C))\cong H_1(C^n,\Q)\cong \bigoplus_{j=1}^nH_1(C_j,\Q),
$$
where $C_j$ is the $j$th copy of $C$. 
Furthermore, the natural MHS on $H^2(F^n(C))$ is pure of weight $2$. 
\end{proposition}
Since $H_1(\p_{g,n})$ is pure of weight $-1$, the natural weight filtration on $\p_{g,n}$ also agrees with its lower central sereis. Denote $H_1(\p_{g,n})=\bigoplus_{j=1}^nH_j$ by $H^{\oplus n}$, where $H_j = H_1(C_j,\Q)$. For $v$ in $H$, denote the corresponding vector of $H_j$ by $v^{(j)}$. Let $\Theta_i = \sum_{l=1}^g[a_l^{(i)}, b_l^{(i)}]$ and $\Theta_{ij} =\sum_{l=1}^g[a_l^{(i)}, b_l^{(j)}]$ for $i\not=j$ in $\L_2(H^{\oplus n})\subset \L(H^{\oplus n})$. 
\begin{theorem}[{\cite[Thm.~12.6]{hain_infini}}]
For $g\geq 1$ and $n\geq 1$, there is an $\Sp(H)$-equivariant isomorphism of graded Lie algebras
$$
\Gr^W_\bullet \p_{g,n}\cong \L(H^{\oplus n})/R,
$$
where $R$ is the Lie ideal generated by the relations
\begin{align*}
[u^{(i)}, v^{(j)}] = [u^{(j)},v^{(i)}] & \text{ for all $i$ and $j$};\\
[u^{(i)}, v^{(j)}] = \frac{\langle u, v\rangle}{g}\Theta_{ij}& \text{ when } i\not =j;\\
\Theta_i + \frac{1}{g}\sum_{j\not =i}\Theta_{ij}=0 & \text{ for } 1\leq i \leq n,
\end{align*}
where $u$ and $v$ are arbitrary vectors in $H$. 
\end{theorem}
As an immediate consequence, we have the following description of $\Gr^W_r\p_{g,n}$ for $r =-1, -2$.
\begin{corollary}\label{pure braid weight -1 and -2}
If $g \geq 2$ and $n\geq 1$, we have
$$
\Gr^W_r\p_{g,n} \cong \begin{cases}
			\bigoplus_{j=1}^n H_j & r =-1\\
			\bigoplus_{j=1}^n\Lambda^2_0H_j \oplus \bigoplus_{1\leq i < j\leq n}\Q(1)_{ij} & r =-2	.
			\end{cases}
$$
\end{corollary}
 From the above description of $\Gr^W_r\p_{g,n}$, we may express the exact sequence 
 $$
 0\to\Gr^W_r \p^o\to \Gr^W_r\p_{g, n+1}\to \Gr^W_r\p_{g,n}\to 0
 $$
 for $r =-1, -2$ as
$$
0\to H_0\to \bigoplus_{j=0}^nH_j \to \bigoplus_{j=1}^nH_j\to 0,
$$
  $$
0\to \Lambda^2_0H_0\oplus \bigoplus_{j=1}^n\Q(1)_{0j}\to \bigoplus_{j=0}^n\Lambda^2_0H_j \oplus \bigoplus_{0\leq i < j\leq n}\Q(1)_{ij}\to  \bigoplus_{j=1}^n\Lambda^2_0H_j \oplus \bigoplus_{1\leq i < j\leq n}\Q(1)_{ij}\to 0,
$$
respectively.

%%%%%%%%%%%%%%%%%%%%%%%%%%%%%%%%%%%%%%%%%%%%%%%%%%%%%%%%%%%%%%%%%%%%%%%%%%%%%%%%%%%%%%%%%%%%%%%%%%%%%%%%%%%%%%%%
\section{Relative completions of hyperelliptic mapping class groups}
In this section, we briefly review the relative completion of a discrete group and the continuous relative completion of a profinite group. Then we introduce the relative completion of hyperelliptic mapping class groups. The Lie algebras of the prounipotent radicals of the completions are the key objects needed in this paper.\\

%%%%%%%%%%%%%%%%%%%%%%%%%%%%%%%%%%%%%%%%%%%%%%%%%%%%%%%%
\subsection{Proalgebraic groups}
By an algebraic group, we mean a linear algebraic group. Suppose that $F$ is a field of characteristic zero. A proalgebraic group $\cG$ over $F$ is an inverse limit of algebraic $F$-groups $G_i$:
$$
\cG = \varprojlim_iG_i.
$$
The Lie algebra $\g$ of $\cG$ is the inverse limit of the Lie algebras $\g_i$ of $G_i$:
$$
\g =\varprojlim_i\g_i.
$$
It is a topological Lie algebra, where  a base of neighborhoods of $0$ consists of the kernels of the natural projections $\g \to \g_i$. 
Similarly, a prounipotent group $\U$ over $F$ is the inverse limit of unipotent $F$-groups $U_i$:
$$
\U =\varprojlim_iU_i.
$$
A pronilpotent Lie algebra $\n$ is the inverse limit of finite dimensional nilpotent Lie algebras $\n_i$:
$$
\n =\varprojlim \n_i.
$$
It is also viewed as a topological Lie algebra. The Lie algebra $\u$ of $\U$ is a pronilpotent Lie algebra, which is the inverse limit of the Lie algebras $\u_i$ of $U_i$. The exponential and log maps are isomorphisms of proalgebraic varieties:
$$
\xymatrix@R=1em@C=2em{        
\U\ar@/^/[r]^{\log}&\ar@/^/[l]^{\exp}\u
}
$$

%%%%%%%%%%%%%%%%%%%%%%%%%%%%%%%%%%%%%%%%%%%%%%%%%%%%%%%%
\subsection{Review of continuous (co)homology}
Suppose that $\n=\varprojlim_i\n_i$ is a pronilpotent Lie algebra over $F$. Define the homology of $\n$ by
$$
H_\bullet(\n):= \varprojlim_iH_\bullet(\n_i).
$$
It is a topological vector space, where the neighborhoods of $0$ are the kernels of the natural projections $H_\bullet(\n)\to H_\bullet(\n_i)$. 
Define the continuous cohomology $H^\bullet(\n)$ by 
$$
H^\bullet(\n):=\varinjlim_iH^\bullet(\n_i).
$$
There are isomorphisms
$$
H_\bullet(\n)\cong \Hom(H^\bullet(\n), F) \text{ and } H^\bullet(\n) \cong \Hom^\cts_F(H_\bullet(\n), F).
$$

%%%%%%%%%%%%%%%%%%%%%%%%%%%%%%%%%%%%%%%%%%%%%%%%%%%%%%%%
\subsection{Relative completion of a discrete group}
Let $F$ be a field of characteristic zero. Suppose that $\G$ is a discrete group, $R$ a reductive linear algebraic group over $F$, and there is a Zariski-dense representation $\rho: \G\to R(F)$. The relative completion of $\G$ over $F$ with respect to $\rho$ consists of a proalgebraic $F$-group $\cG$ that is an extension 
\begin{equation}\label{rel comp seq}
1 \to \U\to \cG\to R\to 1
\end{equation}
of $R$ by a prounipotent $F$-group $\U$ and a Zariski-dense homomorphism $\tilde{\rho}: \G\to \cG$ that makes the diagram
$$
\xymatrix@R=1em@C=2em{
\G\ar[d]_{\tilde\rho}\ar[dr]^{\rho}&\\
\cG\ar[r]&R
}
$$
commute. It satisfies a universal property: If $G$ is an algebraic (resp. a proalgebraic) $F$-group that is an extension
$$
1\to U\to G\to R\to 1
$$ of $R$ by a unipotent (resp. a proalgebraic) $F$-group $U$, and if there is a homomorphism $\rho_G:\G\to G$ lifting $\rho$, then there exists a unique homomorphism $\phi_G: G\to \cG$ such that the diagram
$$
\xymatrix@R=2em@C=2em{
\G\ar[r]^{\rho_G}\ar[d]_{\tilde\rho}&G\ar[d]\ar@{->}[dl]|-{\phi_G}\\
\cG\ar[r]&R
}
$$
commutes.
When $R$ is trivial, we have $\U=\cG$ and $\tilde\rho: \G\to \U$ is the unipotent completion of $\G$ over $F$. Furthermore, by a generalization of Levi's Theorem, the extension (\ref{rel comp seq}) splits, and any two splittings are conjugate by an element of $\U$. Hence the Lie algebra $\u$ of $\U$ is an inverse limit of finite dimensional $R$-modules\footnote{Since $R$ is reductive, this is a direct product of finite dimensional irreducible $R$-modules.} and there is an isomorphism
$$
\cG \cong \exp \u \rtimes R.
$$ 
\indent Relative completions are to some extent computable, since they are controlled by cohomolgy. The following property of relative completion allows us to study the Lie algebra $\u$ of $\U$.
\begin{theorem}[{\cite[Thm.~3.8]{hain_relati}}]\label{rel compl cohom iso}
For all finite dimensional $R$-modules $V$, there is a natural homomorphism
$$
\Hom^\cts_R(H_l(\u), V)\cong(H^l(\u)\otimes V)^R\to H^l(\G, V).
$$
When $l=0, 1$, it is an isomorphism and when $l=2$, it is an injection.  
\end{theorem}
Suppose that $R =\Sp_{2g}(\Q)$. The finite dimensional irreducible representation of $\Sp_{2g}(\Q)$ can be indexed by partitions $\alpha$ of a nonnegative integer $n$ into $\leq g$ parts:
$$
n = \alpha_1+\alpha_2 +\cdots + \alpha_g
$$
with $\alpha_1\geq \alpha_2\geq \cdots \geq \alpha_g\geq 0$. Denote $\alpha$ by $[\alpha_1, \ldots, \alpha_g]$.  Every finite dimensional irreducible representation $V$ of $\Sp_{2g}(\Q)$ is absolutely irreducible, which means that $V\otimes \bar \Q$ is irreducible as an $\Sp_{2g}(\bar \Q)$-representation. 
In this case, we have
\begin{theorem}[{\cite[Thm.~3.8]{hain_relati}}]\label{rel comp iso for sp}
Let $\{V_\alpha\}$ be a set of representatives of the isomorphism classes of finite dimensional irreducible $\Sp_{2g}(\Q)$-modules indexed by partitions $\alpha$. Then there is a natural $\Sp_{2g}(\Q)$-equivariant isomorphism
$$
H^1(\u)\cong \bigoplus_{\alpha}H^1(\G, V_\alpha)\otimes V^\ast_\alpha,
$$
and
a natural $\Sp_{2g}(\Q)$-equivariant injection
$$
H^2(\u)\hookrightarrow \bigoplus_{\alpha}H^2(\G, V_{\alpha})\otimes V^\ast_{\alpha}
$$
\end{theorem}

%%%%%%%%%%%%%%%%%%%%%%%%%%%%%%%%%%%%%%%%%%%%%%%%%%%%%%%%
\subsection{Continuous relative completion of a profinite group}
For the profinite case, let $F$ be the field $\Ql$ for some prime number $\ell$. Suppose that $\G$ is a profinite group, $R$ is a reductive $\Ql$-group, and that $\rho:\G\to R(\Ql)$ is a continuous Zariski-dense representation. 
The continuous relative completion of $\G$ with respect to $\rho$ consists of a proalgebraic $\Ql$-group $\cG$ that is an extension of $R$ by a prounipotent $\Ql$-group $\U$  and a continuous Zariski-dense homomorphism $\tilde\rho: \G\to \cG(\Ql)$ that lifts $\rho$. The pair $(\cG, \tilde\rho)$ also satisfies a universal property that is similar to the one that the relative completion of a discrete group satisfies. All the homomorphisms involved are required to be continuous in this case. When $R$ is trivial, $\tilde\rho$ is the continuous unipotent completion of $\G$ over $\Ql$. \\
\indent  
In this paper, we consider the continuous relative completions of algebraic fundamental groups. Recall that the algebraic fundamental groups we considered are isomorphic to the profinite completions of orbifold fundamental groups. A key fact of relative completion is that it behaves well under the profinite completion of a discrete group. 
A discrete group is considered as a topological group equipped with the smallest topology in which the finite index normal subgroups of $\G$ are open. The following result allows us to compute the continuous completion from the relative completion of the discrete group.
\begin{theorem}[{\cite[Thm.~ 6.3]{hain_ration}}]
Suppose that $\G$ is a discrete group, $R$ is a reductive $\Ql$-group, and that $\rho:\G\to R(\Ql)$ is a continuous, Zariski-dense representation. Let $\rho_\ell: \widehat\G\to R(\Ql)$ be the unique continuous homomorphism induced by the profinite completion of $\G$ extending $\rho$. If $\cG$ and $\tilde\rho:\G\to \cG(\Ql)$ is the relative completion of $\G$ with respect to $\rho$, then 
\begin{enumerate}
\item $\tilde\rho$ is continuous and so induces a unique extension $\hat\rho_\ell: \widehat\G\to \cG(\Ql)$, and
\item the pair of $\cG$ and $\hat\rho_\ell$ is the continuous relative completion of $\widehat\G$ with respect to $\rho_\ell$.
\end{enumerate}

\end{theorem}

%%%%%%%%%%%%%%%%%%%%%%%%%%%%%%%%%%%%%%%%%%%%%%%%%%%%%%%%
\subsection{Applications to the universal hyperelliptic curves}\label{app to hyp univ curve}
Suppose that $g \geq 2$, $n\geq 0$, and $m\geq 1$.  Consider the universal complete hyperelliptic curve $\pi:\cC_{\cH_{g,n}}[m]\to \cH_{g,n}[m]$. Let $\H_\Z$ be the dual of the local system $R^1\pi_\ast\Z$ and $H_\Z$ the fibre  of $\H_\Z$ over $\etabar_n$.  Set $H:=H_\Q =H_\Z\otimes \Q$. We have an isomorphism $\pi_1(\cH_{g,n}[m], \etabar_n)\cong \Delta_{g,n}[m]$, and there is the monodromy representation 
$$
\rho^\hyp[m]: \pi_1(\cH_{g,n}[m], \etabar_n)\to \Sp(H_\Z), 
$$
which agrees with the restriction to $\Delta_{g,n}[m]$ of the natural representation $\rho^\hyp:\Delta_{g,n}\to \Sp(H_\Z)$ defined in \S \ref{monodromy}. When $m=1$, we denote $\rho^\hyp[m]$ by $\rho^\hyp$. Recall that $G_g$, the image of $\rho^\hyp$, contains the principal congruence subgroup $\Sp(H_\Z)[2]$, and hence it is a finite-index subgroup of $\Sp(H_\Z)$, and  so is the image of $\rho^\hyp[m]$. Therefore, the image of $\rho^\hyp[m]$ is Zariski-dense in $\Sp(H)$.  Denote the representation $\pi_1(\cH_{g,n}[m])\to \Sp(H)$ obtained by extending the coefficient to $\Q$ by $\rho^\hyp[m]$ as well. \\
\indent Denote the relative completion of $\pi_1(\cH_{g,n}[m], \etabar_n)$ with respect to $\rho^\hyp[m]$ by $\cD^\geom_{g,n}[m]$ and its prounipotent radical by $\cV^\geom_{g,n}[m]$. Let $\rho^\hyp_\cC[m]:\pi_1(\cC_{\cH_{g,n}}[m])\to\Sp(H)$ be the representation obtained by composing $\rho^\hyp[m]$ with $\pi_\ast:\pi_1(\cC_{\cH_{g,n}}[m])\to \pi_1(\cH_{g,n}[m])$. Since $\pi_\ast$ is surjective, $\rho^\hyp_\cC[m]$ is a Zariski-dense representation. Denote the relative completion of $\pi_1(\cC_{\cH_{g,n}}[m])$ with respect to $\rho^\hyp_{\cC}[m]$ by $\cD^{\geom}_{\cC_{g,n}}[m]$ and its prounipotent radical by $\cV^\geom_{\cC_{g,n}}[m]$. 
Similarly, let $\rho^\hyp_{\cC'}[m]:\pi_1(\cC_{\cH_{g,n}}'[m])\to \Sp(H)$ be the representation obtained by composing $\rho^\hyp[m]$ with $\pi'_\ast: \pi_1(\cC_{\cH_{g,n}}'[m])\to \pi_1(\cH_{g,n}[m])$. Denote the relative completion of $ \pi_1(\cC_{\cH_{g,n}}'[m])$ with respect to $\rho^\hyp_{\cC'}[m]$ by $\cD^\geom_{\cC'_{g,n}}[m]$ and its prounipotent radical  by $\cV^\geom_{\cC'_{g,n}}[m]$. Finally, let $\widehat\pi:\cC^{n}_{\cH_{g}}[m]\to \cH_g[m]$ be the $n$th power of the universal curve $\cC_{\cH_{g}}[m]\to \cH_g[m]$. Denote by $\widehat\cD^\geom_{g, n}[m]$ the relative completion of $\pi_1(\cC^{n}_{\cH_g}[m])$ with respect to the composition of $\rho^\hyp[m]$ with $\widehat \pi_\ast:\pi_1(\cC^{n}_{\cH_g}[m])\to \pi_1(\cH_g[m])$ and its prounipotent radical by $\widehat\cV^\geom_{g,n}[m]$. 
Denote the Lie algebras of $\cD^\geom_{g,n}[m]$, $\cV^\geom_{g,n}[m]$, $\cD^\geom_{\cC_{g,n}}[m]$, $\cV^\geom_{\cC_{g,n}}[m]$, $\cD^\geom_{\cC'_{g,n}}[m]$, $\cV^\geom_{\cC'_{g,n}}$, $\widehat\cD^\geom_{g, n}$, and $\widehat\cV^\geom_{g,n}$ by $\d^\geom_{g,n}[m]$, $\v^\geom_{g,n}[m]$, $\d^\geom_{\cC_{g,n}}[m]$, $\v^\geom_{\cC_{g,n}}[m]$, $\d^\geom_{\cC'_{g,n}}[m]$, $\v^\geom_{\cC'_{g,n}}$,  $\widehat\d^\geom_{g, n}[m]$, and $\widehat\v^\geom_{g,n}[m]$, respectively. \\

%%%%%%%%%%%%%%%%%%%%%%%%%%%%%%%%%%%%%%%%%%%%%%%%%%%%%%%%%%%%%
\subsection{Continuous completion}\label{cont hyp rep} Let $\ell$ be a prime number. Let $\H_\Ql$ be the sheaf $R^1\pi_\ast\Ql(1)$ over $\cH_{g,n/\C}[m]$. Note that the fiber of $\H_\Ql$ over $\etabar_n$ is isomorphic to $H_\Ql =H_1(C,\Z)\otimes \Ql$. Then we have the monodromy representation 
$$
\rho^\hyp_\ell[m]:\pi_1^\alg(\cH_{g,n/ \C}[m], \etabar_n)\to \Sp(H_\Ql).
$$
Consider $\rho^\hyp[m]$ in \S \ref{app to hyp univ curve} as a homomorphism into $\Sp(H_\Ql)$ by extending the coefficient to $\Ql$. We observe that the homomorphism $\rho^\hyp_\ell[m]$ agrees with the continuous extension of $\rho^\hyp[m]$ to $\pi_1^\alg(\cH_{g,n/\C}[m])\cong \widehat{\pi_1(\cH_{g,n}[m])}$. The continuous relative completion of  $\pi_1^\alg(\cH_{g,n/\C}[m],\etabar_n)$ with respect to $\rho^\hyp_\ell[m]$ is $\cD^\geom_{g,n}[m]\otimes_\Q \Ql$. Similarly, the continuous relative completions of $\pi_1(\cC_{\cH_{g,n/\C}}[m], \etabar_{n+1})$, $\pi_1(\cC_{\cH_{g,n/\C}}'[m], \etabar_{n+1})$,  and $\pi_1(\cC^n_{\cH_{g/\C}}[m], \etabar_n)$ are obtained from $\d^\geom_{\cC_{g,n}}[m]$, $\d^\geom_{\cC'_{g,n}}[m]$, and $\widehat\d^\geom_{g,n}[m]$ by base change to $\Ql$, respectively.

%%%%%%%%%%%%%%%%%%%%%%%%%%%%%%%%%%%%%%%%%%%%%%%%%%%%%%%%%%%%%
\subsection{Exact sequences of completions} 
\indent Let $\cP_j$ and $\cP'$  be the unipotent completions of $\pi_1(C, \bar x_j)$ and $\pi_1(C', \bar x_0)$ over $\Q$. Denote the Lie algebras of $\cP_j$ and  $\cP'$ by  $\p_j$ and $\p'$, respectively. The Lie algeras $\p$, $\p_j$, and $\p_{g,n}$ have trivial center \cite{ntu} and the Lie algebras $\p'$ and $\p^o$ are free and therefore have trivial center. 
The center-freeness of these Lie algebras and the right exactness of relative completion \cite[Prop.~3.7]{hain_relati} imply that the exact sequences (\ref{punct homo seq}), (\ref{1 punct homo seq}), (\ref{complete homo seq}), (\ref{n+1th power over nth seq }), (\ref{univ seq with conf fiber}), and (\ref{nth power seq}) induce the exact sequences of Lie algebras of completions: %\textcolor{red}{rewrite in terms of $\v_g$}
\begin{equation*}
0\to \p^o\to \d^\geom_{g, n+1}[m]\to \d^\geom_{g,n}[m]\to 0,
\end{equation*}
\begin{equation*}
0\to \p'\to \d^\geom_{\cC'_{g,n}}[m]\to \d^\geom_{g,n}[m]\to 0, 
\end{equation*}
\begin{equation*}
0\to \p\to \d^\geom_{\cC_{g,n}}[m]\to \d^\geom_{g,n}[m]\to 0, 
\end{equation*}
\begin{equation*}
0\to \p\to \widehat\d^\geom_{g,n+1}[m]\to \widehat\d^\geom_{g,n}[m]\to 0, 
\end{equation*}
\begin{equation*}
0\to\p_{g,n}\to \d^\geom_{g,n}[m]\to \d^\geom_g[m]\to 0,
\end{equation*}
and 
\begin{equation*}
0\to\bigoplus_{j=1}^n\p_j\to\widehat\d^\geom_{g,n}[m]\to \d^\geom[m]\to 0.
\end{equation*}
Therefore, the commutative diagrams (\ref{comm diag fund grp}) and (\ref{comm diag fund grp with the nth power}) produce the following commutative diagrams of Lie algebras of completions:
\begin{proposition}\label{commu diagram for lie alg}
Suppose that $g \geq 2$. Applying relative completion to the diagrams (\ref{comm diag fund grp}) and (\ref{comm diag fund grp with the nth power}), we obtain the commutative diagrams:
$$
\xymatrix@R=1em@C=2em{
0\ar[r]&\p^o\ar[r]\ar@{=}[d]&\p_{g, n+1}\ar[r]\ar@{^{(}->}[d]  &\p_{g,n}\ar[r]\ar@{^{(}->}[d]        &0\\
0\ar[r]&\p^o\ar[r]  \ar[d] &\d^\geom_{g,n+1}[m]\ar[r]\ar[d] & \d^\geom_{g,n}[m]\ar[r]\ar@{=}[d]  &0\\
0\ar[r]&\p'\ar[r]\ar[d] &\d^\geom_{\cC'_{g,n}}[m]\ar[r]\ar[d]    & \d^\geom_{g,n}[m]\ar[r]\ar@{=}[d]&0\\
0\ar[r]&\p\ar[r]\ar@{=}[d] &\d^\geom_{\cC_{g,n}}[m]\ar[r]\ar[d]    & \d^\geom_{g,n}[m]\ar[r]\ar[d]&0\\
0\ar[r]&\p\ar[r]  &\widehat\d^\geom_{g,n+1}[m]\ar[r]    & \widehat\d^\geom_{g,n}[m]\ar[r]&0
}
$$
and 
$$
\xymatrix@R=1em@C=2em{
0\ar[r]&\p_{g,n}\ar[r]\ar[d]& \d^\geom_{g,n}[m]\ar[r]\ar[d]  &\d^\geom_g[m]\ar[r]\ar@{=}[d]        &0\\
0\ar[r]&\bigoplus_{j=1}^n\p_j\ar[r]  \ar@{=}[d]      &\widehat\d^\geom_{g,n}[m]\ar[r]\ar[d] & \d^\geom_g[m]\ar[r]\ar[d] ^{\mathrm{diag}} &0\\
0\ar[r]&\bigoplus_{j=1}^n\p_j\ar[r]                    &\bigoplus_{j=1}^n\d^{\geom(j)}_{g,1}[m]\ar[r]    & (\d^\geom_{g}[m])^n\ar[r]&0
}
$$
where rows are exact and $\d^{\geom(j)}_{g,1}[m]$ is the Lie algebra of the completion of $\pi_1(\cH_{g,1}[m], [C;\bar x_j])$.  The same result holds when replacing $\d^\geom_g[m]$, $\d^\geom_{g,n}[m]$, $\d^\geom_{\cC'g,n}[m]$, $\d^\geom_{\cC_{g,n}}[m]$, and $\widehat\d^\geom_{g,n}[m]$ with $\v^\geom_g[m]$, $\v^\geom_{g,n}[m]$, $\v^\geom_{\cC'g,n}[m]$, $\v^\geom_{\cC_{g,n}}[m]$, and $\widehat\v^\geom_{g,n}[m]$, respectively.  %\textcolor{red}{Define $\d^\geom_{g,1}[m]_j$ before the prop.}
\qed
\end{proposition}

%%%%%%%%%%%%%%%%%%%%%%%%%%%%%%%%%%%%%%%%%%%%%%%%%%%%%%%%
\subsection{Weight filtrations on the Lie algebras of relative completions}
Suppose that $F =\Q$ or $\mathbb{R}$. Suppose that $X$ is a smooth quasi-projective complex variety and that $\V$ is a polarized variation of $F$-HS over $X$ of geometric origin.  Denote by $R$ the Zariski closure of the image of the monodromy representation  
	$$\rho_x:\pi_1(X,x)\to \Aut(\V_x, \langle~,~\rangle).$$ It is a reductive group over $F$. Denote the relative completion of $\pi_1(X, x)$ with respect to $\rho_x:\pi_1(X,x)\to R(F)$ by $\cG$ 
	and the prounipotent radical of $\cG$ by $\U$. Denote the Lie algebras of $\cG$, $\U$, and $R$ by $\g$, $\u$, and $\r$, respectively. 
\begin{theorem}[{\cite[Thm.~13.1]{hain_hodge_rel}}]\label{hs on relative completion}
The coordinate rings $\O(\cG)$ and $\O(\U)$ of $\cG$ and its prounipotent radical $\U$, respectively, are Hopf algebras that admit natural $F$-MHSs satisfying the property that $W_{-1}\O(\cG) = 0$ and $W_0\O(\cG) = \O(R)$. Consequently, the Lie algebras $\g$ and $\u$ admit natural $F$- MHSs, where brackets are morphisms of MHSs. These Lie algebras satisfy the property that
$$
\g = W_0\g, \hspace{.1in} \u = W_{-1}\g, \hspace{.1in}  \Gr^W_0\g =\r.
$$
\end{theorem}
\begin{remark}
	The original statement is stated over $\mathbb{R}$. However, if the Zariski-closure of the image of $\rho_x$ is defined over $\Q$, the $\mathbb{R}$-structure canonically lifts to a $\Q$-structure. For a concrete explanation, see  \cite[4.3]{hain_hodge_modular}. Furthermore, the results also extend for orbifolds. 
\end{remark}
\begin{corollary} \label{rel comp mhs iso}
If $\V$ is a polarized variation of Hodge structure (PVHS) over $X$ with fiber $V_x$ over the base point $x$, then the composite
$$
(H^l(\u)\otimes V_x)^R\to H^l_\cts(\pi_1(X, x), V_x)\to H^l(X,\V)
$$
with the canonical homomorphism is a morphism of MHS. It is an isomorphism for $l\leq 1$ and an injection for $l =2$. 
\end{corollary}
The adjoint map $\ad: \g\to \Der \g$ is a morphism of MHSs, where $\Der\g=W_0\Der\g$ is the derivation algebra of $\g$. Since the sequence 
$$
0\to\u\to \g\to\r\to0
$$
is exact, it follows that each graded quotient $\Gr^W_l\g$ is a direct product finite dimensional $\r$-modules. 
Denote the Lie algebra of $\Sp(H)$ by $\s\p(H)$. The representation theory of $\Sp(H)$ and $\s\p(H)$ are equivalent.
\begin{corollary} Suppose that $g \geq 2$, $n\geq 0$,  and $m\geq 1$. 
 The Lie algebra $\d^\geom_{g,n}[m]$ is a MHS, and satisfies the property:
$$
\d^\geom_{g,n}[m] =W_0\d^\geom_{g,n}[m], \,\,\v^\geom_{g,n}[m] = W_{-1}\d^\geom_{g,n}[m] \text{ and }\Gr^W_0\d^\geom_{g,n}[m] = \s\p(H).
$$
%and 
%$$
%\d^\cC_{g,n} =W_0\d^\cC_{g,n}, \,\,\v^\cC_{g,n} = W_{-1}\d^\cC_{g,n} \text{ and }\Gr^W_0\d^\cC_{g,n} = \s\p(H).
%$$
 Each graded quotient $\Gr^W_m\d^\geom_{g,n}$ is a direct product of finite dimensional $\Sp(H)$-modules. The similar property holds for $\d^\geom_{\cC_{g,n}}[m]$, $\d^\geom_{\cC'_{g,n}}[m]$, and $\widehat\d^\geom_{g,n}[m]$
\end{corollary}
Furthermore, by the naturality properties \cite[Thm.~13.12]{hain_infini} of the MHS of relative completions, the projections $$
\d^\geom_{g, n+1}[m]\to \d^\geom_{g,n}[m],\,\,\,\, \d^\geom_{\cC'_{g,n}}[m]\to \d^\geom_{g,n}[m],\,\,\,\, \d^\geom_{\cC_{g,n}}[m]\to \d^\geom_{g,n}[m],  
$$
 $$
\widehat\d^\geom_{g,n+1}\to\widehat\d^\geom_{g,n},\,\,\,\,\d^\geom_{g,n}\to \d^\geom_g,\text{ and } \widehat\d^\geom_{g,n}[m]\to \d^\geom_g[m]
$$
 are morphisms of MHSs. \\
\indent Recall that the Lie algebras $\p$, $\p'$, $\p^o$, and $\p_{g,n}$ 
%of the unipotent completions $\cP$, $\cP'$, $\cP^o$, and $\cP_{g,n}$, respectively, 
admit natural weight filtrations as MHSs constructed by Morgan in \cite{morgan} and Hain in \cite{hain_deRham}.  On the other hand, these Lie algebras also admit weight filtrations being the kernel of the natural surjections given in Proposition \ref{commu diagram for lie alg}. By the naturality properties \cite[Thm.~13.12]{hain_infini} of the MHS of relative completions, these weight filtrations agree. Therefore, the adjoint maps
$$
\d^\geom_{g, n+1}[m]\to \Der\p^o,\,\, \d^\geom_{\cC g, n}[m]\to \Der \p,\,\, \widehat\d^\geom_{g, n+1}[m]\to \Der \p\text{ and } \d^\geom_{g,n}[m]\to \Der\p_{g,n}
$$
are morphisms of MHSs with respect to the MHSs determined by relative completion.

%%%%%%%%%%%%%%%%%%%%%%%%%%%%%%%%%%%%%%%%%%%%%%%%%%%%%%%%%%%%%
\subsection{The abelianization $H_1(\v^\geom_g[m])$} For each partition $\alpha$, the irreducible representation $V_\alpha$ of $\Sp(H)$ defines a HS of weight $-|\alpha|$, where $|\alpha|:=\sum_{i=1}^g\alpha_i$.  When $g=2$, it follows from \cite[Thm.~5.1]{petersen} that $H_1(\v^\geom_2)$ is pure of weight $-2$, and that it is generated by $V_{2+2}(-1)$. In general, we have
\begin{proposition}\label{pure weight -2}
If $g\geq 2$ and $m\geq 1$, $H_1(\v^\geom_g[m])$ is pure of weight $-2$. 
\end{proposition}
\begin{proof}
Suppose that $m\geq 3$. In this case, $\cH^{c}_g[m]$ and $\cH_g[m]$ are smooth complex varieties. Let $V_\alpha$ be a finite dimensional irreducible representation of $\Sp(H)$ corresponding to a partition $\alpha$. Let $\V_\alpha$ be the VHS over $\cH^{c}_g[m]$ corresponding to $V_\alpha$. The complement of $\cH_g[m]$ in $\cH^{c}_g[m]$ is a divisor, which we denote by $D :=\cH^{c}_g[m] - \cH_g[m]$. Let $D^\mathrm{sing}$ be the singular locus of $D$. It has codimension at least $2$ in $\cH^{c}_g[m]$.  Let $D^\mathrm{sm}=D-D^\mathrm{sing}=\cup_iD_i$, where $D_i$ is an irreducible component of $D^\mathrm{sm}$. Then associated to the inclusion $\cH_g[m]\to \cH^{c}_g[m]$, there is a Gysin sequence 
$$
0\to H^1(\cH^{c}_g[m],\V_\alpha)\to H^1(\cH_g[m],\V_\alpha)\to H^0(D^\mathrm{sm},\V_\alpha(-1))=\bigoplus_{i}H^0(D_i,\V_\alpha(-1)).
$$
Since $\cH^{c}_g[0]$ is simply-connected by a result in \cite{BMP}, it follows that there is an isomorphism
$$
H^1(\cH^{c}_g[m],\V_\alpha)\cong H^1(G_g[m],V_\alpha).
$$
Raghunathan's vanishing theorem \cite{raghu} implies that $H^1(G_g[m], V_\alpha) =0$, and hence the cohomology $H^1(\cH_g[m],\V_\alpha)$ is a HS of weight $2 -|\alpha|$. By Corollary \ref{rel comp mhs iso}, $H_1(\v^\geom_g[m])$ is pure of weight $-2$. \\
\indent For the case $m\leq 2$, suppose that  $m_1, m_2$ are  positive integers such that $m_1\leq 2$, $m_2\geq 3$,  and $m_1|m_2$.   Let $Q[m_2/m_1]$ be the finite quotient of $\Delta_g[m_1]$ by $\Delta_g[m_2]$. From the spectral sequence associated to the exact sequence 
$$
1 \to \Delta_g[m_2]\to \Delta_g[m_1] \to Q[m_2/m_1]\to 1,
$$ 
it follows that the restriction map induces a natural isomorphism
$$
H^1(\Delta_g[m_1],V_\alpha)\to H^1(\Delta_g[m_2], V_\alpha)^{Q[m_2/m_1]}.
$$
By Theorem \ref{rel comp iso for sp}, there is a commutative diagram
$$
\xymatrix@R=2em@C=2em{
H^1(\v^\geom_g[m_1])\ar[d]_\cong\ar[r]&H^1(\v^\geom_g[m_2])\ar[d]^\cong\\
\bigoplus_{\alpha}H^1(\Delta_g[m_1], V_\alpha)\otimes V_\alpha^\ast\ar[r]&\bigoplus_{\alpha}H^1(\Delta_g[m_2], V_\alpha)\otimes V_\alpha^\ast,
}
$$
where the top map is induced by the relative completion applied to $\Delta_g[m_2]\to \Delta_g[m_1]$. 
The bottom map is injective, and so is the top map. Therefore, by taking the full dual, the natural map $H_1(\v^\geom_g[m_2])\to H_1(\v^\geom_g[m_1])$ is surjective. This surjection is a morphism of MHS, and hence $H_1(\v^\geom_g[m_1])$ is also pure of weight $-2$. 
\end{proof}
As a consequence, we obtain the following description of $\Gr^W_{\bullet}\v^\geom_g[m]$.
\begin{corollary}\label{abelian weight -2}
If $g \geq 2$ and $m\geq 1$, then $\Gr^W_{-2l-1}\v^\geom_g[m]=0$ for $l\geq 0$ and
$$
\Gr^W_{-2}\v^\geom_g[m] =H_1(\v^\geom_g[m])
$$
\end{corollary}
\begin{proof} The exactness of the functor $\Gr^W_\bullet$ implies that homology commutes with $\Gr^W_\bullet$. 
Therefore,  we have canonical $\Sp(H)$-module isomorphisms
$$
H_1(\Gr^W_\bullet\v^\geom_g[m]) = \Gr^W_\bullet H_1(\v^\geom_g[m]) = H_1(\v^\geom_g[m])
$$ 
by Proposition \ref{pure weight -2}. Since $\Gr^W_\bullet\v^\geom_g[m]$ is a graded Lie algebra with negative weights, by \cite[Prop.~5.6 \& Cor.~5.7]{hain_infini} there is a graded Lie algebra surjection from the free Lie algebra generated by $H_1(\Gr^W_\bullet\v^\geom_g[m])$:
$$
\L(H_1(\Gr^W_\bullet\v^\geom_g[m]))=\L(H_1(\v^\geom_g[m]))\to \Gr^W_\bullet\v^\geom_g[m].
$$
Therefore,  $\Gr^W_\bullet\v^\geom_g[m]$ is generated by $H_1(\v^\geom_g[m])$, which is pure of weight $-2$, and so  we have $\Gr^W_{-2l-1}\v^\geom_g[m]=0$ for $l\geq 0$. In particular, it follows that $\v^\geom_g[m] = W_{-2}\v^\geom_g[m]$. 
%Furthermore, the map $\psi$ induces a surjection $\Gr^W_{-2}\psi:H_1(\v^\geom_g[m])\to \Gr^W_{-2}\v^\geom_g[m]$. On the other hand, the abelianization map $\v^\geom_g[m]\to H_1(\v^\geom_g[m])$ yields a surjection 
%$$
%\Gr^W_{-2}\v^\geom_g[m]\to \Gr^W_{-2}H_1(\v^\geom_g[m])=H_1(\v^\geom_g[m]).
%$$
%The composition 
%$$
%H_1(\v^\geom_g[m])\overset{\Gr^W_{-2}\psi}\to \Gr^W_{-2}\v^\geom_g[m]\to H_1(\v^\geom_g[m])
%$$
% is an isomorphism, and hence $\Gr^W_{-2}\psi$ is an isomorphism. 
Now, consider the exact sequence
$$
0\to \Gr^W_{-2}\v^\geom_g[m]'\to \Gr^W_{-2}\v^\geom_g[m]\to \Gr^W_{-2}H_1(\v^\geom_g[m])\to 0,
$$
where $\v^\geom_g[m]'=[\v^\geom_g[m], \v^\geom_g[m]]$. Since $\Lambda^2\v^\geom_g[m] \to \v^\geom_g[m]'$ is a surjection of MHSs, we have $\v^\geom_g[m]' = W_{-4}\v^\geom_g[m]'$. Thus we have $\Gr^W_{-2}\v^\geom_g[m]'=0$, and hence 
$\Gr^W_{-2}\v^\geom_g[m]= \Gr^W_{-2}H_1(\v^\geom_g[m])=H_1(\v^\geom_g[m])$. 

\end{proof}
When $m=1$, Tanaka's computations on the cohomology of $\Delta_g$ with twisted coefficients give us the following result. 
\begin{proposition}\label{tanaka's computation}
If $g \geq 2$, then $\Hom^\cts_{\Sp(H)}(H_1(\v^\geom_g), V_{\alpha}) =0$ for $\alpha = [0]$, $[1]$, or $[1^2]$. 
\end{proposition}
\begin{proof}
By Theorem \ref{rel comp iso for sp}, there is an isomorphism $\Hom^\cts_{\Sp(H)}(H_1(\v^\geom_g), V_{\alpha})\cong H^1(\Delta_g, V_{\alpha})$. Since $H_1(\Delta_g, \Z)$ is torsion by \cite[Thm.~8]{birman-hilden}, we have $H^1(\Delta_g,\Q)=0$. On the other hand, it follows from Tanaka's computations \cite[Thm.~1.1 \& Thm.~1.3]{tanaka} that $H^1(\Delta_g, H)=0$ and $H^1(\Delta_g,  \Lambda^2_0H)=0$. Therefore, our claim follows. 
\end{proof}
%%%%%%%%%%%%%%%%%%%%%%%%%%%%%%%%%%%%%%%%%%%%%%%%%%%%%%%%%%%%%
\subsection{Applications to the universal curve $\cC_{g,n}[m]\to \M_{g,n}[m]$} We review the relative completion of the mapping class groups here. For a detailed introduction, see \cite{hain_compl} or \cite{hain_infini}. 
The natural monodromy representation $\pi_1(\M_{g,n}[m])\to \Sp(H)$ can be identified with the homomorphism $\rho[m]: \G_{g,n}[m]\to \Sp(H)$. 
Denote the relative completion of $\G_{g,n}[m]$ with respect to $\rho[m]$ by $\cG^\geom_{g,n}[m]$ and its prounipotent radical by $\U^\geom_{g,n}[m]$. Similarly, denote the relative completion of $\pi_1(\cC_{g,n}[m])$ with respect to the monodromy representation to  $\Sp(H)$ by $\cG^\geom_{\cC_{g,n}}[m]$ and its prounipotent radical by $\U^\geom_{\cC_{g,n}}$[m]. In the case of the mapping class groups, the relative completion does not depend on the level $m$ when $g\geq 3$ (see \cite[Prop.~3.3]{hain_infini}). Hence we drop the level structure $m$ from the completions. Let $\u^\geom_{g,n}$ and $\u^\geom_{\cC_{g,n}}$ be the Lie algebras of $\U^\geom_{g,n}$ and $\U^\geom_{\cC_{g,n}}$, respectively. \\
\indent The polarization $\Lambda^2H\to\Q(1)$ is viewed as an element $\check{\theta}$ of $\Lambda^2H(-1)$.   Denote  the quotient $\Lambda^3H(-1)/H\wedge \check{\theta}$ by $\Lambda^3_0H$. The $\Sp(H)$-module $\Lambda^3_0H$ is isomorphic to the irreducible representation $V_{[1^3]}$ of $\Sp(H)$. 
 In \cite{hain_ration}, the weight filtrations on $\u^\geom_{g,n}$ was constructed by weighted completion. The Lie algebra $\u^\geom_{g,n}$ also admits a natural weight filtration from the Hodge theory by \cite[Thm.~13.1]{hain_hodge_rel}. Since  $H_1(\u^\geom_{g,n})$ is pure of weight $-1$ in both cases, these weight filtrations agree with the lower central series of $\u^\geom_{g,n}$. We have the following description for $H_1(\u^\geom_{g,n})$.
 \begin{theorem}[{\cite[Thm.~9.11]{hain_ration}}]\label{ugn ab}
 If $g\geq 3$ and and $n\geq 0$, then  there are $\Sp(H)$-equivariant isomorphisms
$$
H_1(\u^\geom_{g,n})\cong \Gr^W_{-1}\u^\geom_{g,n}\cong \Lambda^3_0H\oplus \bigoplus_{j=1}^nH_j.
$$
 \end{theorem}
 Applying relative completion to the homotopy exact sequence
 $$
 1\to \pi_1(C, \bar x_0)\to \pi_1(\cC_{g,n}[m], \etabar_{n+1})\to \pi_1(\M_{g,n}[m], \etabar_n)\to 1
 $$ produces the exact sequence of pronilpotent Lie algebras
 $$
 0\to\p\to \u^\geom_{\cC_{g,n}}\to \u^\geom_{g,n}\to 0. 
 $$
 From this sequence, it follows that  there are $\Sp(H)$-equivariant isomorphisms
$$
H_1(\u^\geom_{\cC_{g,n}})\cong \Gr^W_{-1}\u^\geom_{\cC_{g,n}}\cong \Lambda^3_0H\oplus \bigoplus_{j=0}^nH_j.
$$
%%%%%%%%%%%%%%%%%%%%%%%%%%%%%%%%%%%%%%%%%%%%%%%%%%%%%%%%%%%%%%%%%%%%%%%%%%%%%%%%%%%%%%%%%%%%%%%%%%%%%%%%%%%%%%%%
\section{Weighted completion of $\pi^\alg_1(\cH_{g,n/k})$} The weighted completion $\cG$ of a profinite group $\G$ is the key tool used in \cite{wat_sec}. Weighted completion is a variant of relative completion and introduced by Hain and Matsumoto in \cite{HaMa_wcomp}.  It can be used to define weight filtrations with strong exactness properties on $\cG$-modules and it is also computable, since it is controlled by cohomology. The proof of Theorem \ref{sections over C} is reduced to the case considered in \cite[Thm.~1]{wat_sec}, and for this reason, we will briefly review the definition of weighted completion and define the weighted completion of $\pi_1^\alg(\cH_{g,n/k})$ in this section. \\
\indent  Suppose that the following data: $\G$ is a profinite group, $R$ is a reductive group over $\Ql$, $\omega: \Gm\to R$ is a central cocharacter, and $\rho: \G\to R(\Ql)$ is a continuous homomorphism with Zariski-dense image. An extension $G$ 
$$
1 \to U\to G\to R\to 1
$$
 of $R$ by a prounipotent group $U$ is said to be a negatively weighted extension if $H_1(U)$ admits only negative weights as a $\Gm$-module via the central cocharacter $\omega$. The weighted completion of $\G$ with respect to $\rho$ and $\omega$ consists of a proalgebraic group $\cG$ over $\Ql$ that is a negatively weighted extension of $R$ by a prounipotent group $\U$ and a Zariski-dense homomorphism $\tilde\rho: \G\to \cG(\Ql)$ that satisfies a universal property: If there is a negatively weighted extension $G$ of $R$ and there is a homomorphism $\rho_G: \G\to G$ lifting $\rho$, then there exists a unique homomorphism $\psi_G: G\to \cG$ such that $\tilde\rho =\psi_G\circ \rho_G$. By a generalization of Levi's theorem, the extension
 $$
 1\to \U\to \cG\to R\to 1
 $$
splits and any two splittings are conjugate by an element of $\U(\Ql)$. Denote the Lie algebras of $\cG$, $\U$, and $R$ by $\g$, $\u$, and $\r$, respectively. By fixing a splitting $s: R\to \cG$, the adjoint action of $\cG$ on $\g$ produces a natural weight filtration $\mathcal{W}_\bullet \g$ satisfying that $\g =\mathcal{W}_0\g$, $\u = \mathcal{W}_{-1}\u = \mathcal{W}_{-1}\g$, and $\Gr^\mathcal{W}_0\g =\r$.  The adjoint action of $\cG$ on $\Gr^\mathcal{W}_r\g$ factors through $R$.  Furthermore, the functor $\Gr^\mathcal{W}_\bullet$ is exact in the category of finite dimensional (and hence pro-and ind-)$\cG$-modules. \\
\indent Suppose that $k$ is a field of characteristic zero such that the image of the $\ell$-adic cyclotomic character $\chi_\ell: G_k\to \Z^\times$ of the absolute Galois group of $k$ is infinite. Let $R =\GSp(H_\Ql)$.  Define a central cocharacter $\omega:\Gm\to R$ by sending $\lambda\mapsto \lambda^{-1}\id$. Consider  the monodromy representation $\rho^\arith_\ell[m]: \pi_1^\alg(\cH_{g,n/k}[m], \etabar_n)\to \GSp(H_\Ql)$. The Zariski-density of the monodromy $\rho^\hyp_\ell[m]:\pi_1^\alg(\cH_{g,n/\bar k}[m], \etabar_n)\to \Sp(H_\Ql)$ implies that $\rho^\arith_\ell[m]$ has Zariski-dense image. 
Denote the weighted completion of $\pi_1^\alg(\cH_{g,n/k}, \etabar_n)$ with respect to $\rho^\arith_\ell[m]$ and $\omega$ by $\cD_{g,n}[m]$.  It is a proalgebraic group over $\Ql$ and its prounipotent radical is denoted by $\cV_{g,n}[m]$. Denote the Lie algebras of $\cD_{g,n}[m]$ and $\cV_{g,n}[m]$ by $\d_{g,n}[m]$ and $\v_{g,n}[m]$. 
%%%%%%%%%%%%%%%%%%%%%%%%%%%%%%%%%%%%%%%%%%%%%%%%%%%%%%%%%%%%%%%%%%%%%%%%%%%%%%%%%%%%%%%%%%%%%%%%%%%%%%%%%%%%%%%%
\section{The classes of sections of $\pi: \cC_{\cH_{g,n}}[m]\to \cH_{g,n}[m]$}
Let $k\subset\C$ be a field. For a geometrically connected stack $\X_{/k}$ over $k$, let $f:C\to \X_{/k}$ be a curve of genus $g$ with $g\geq 2$.
 Denote the canonical divisor class of $f$ by $\omega_f$. For each section $x$ of $f$, the cycle $(2g-2)x - \omega_f$ is a relative cycle of codimension zero and homologous to zero on each fiber of $f$. Hence, the topological Abel-Jacobi map (see \cite[\S 4]{hain_matsu}) associates to the cycle a characteristic class $\kappa_x$ in $H^1(\X, \H_\Z)$. \\
 \indent  For $j=1,\ldots, n$, denote by $\kappa_j$ the characteristic class in $H^1(\M_{g,n}, \H_\Z)$ induced by the $j$th tautological section $x_j$ of $\cC_{g,n/k}\to \M_{g,n/k}$. Recall that the sections $s_j$ of the universal hyperelliptic curve $\pi:\cC_{\cH_{g,n/k}}[m]\to \cH_{g,n/k}[m]$ are the pullbacks of the sections $x_j$. The class in $H^1(\cH_{g,n}[m], \H_\Z)$ corresponding to the section $s_j$ of $\pi$ is denoted by $\kappa^\hyp_{j}$. 
%%%%%%%%%%%%%%%%%%%%%%%%%%%%%%%%%%%%%%%%%%%%%%%%%%%%%%%%
\subsection{The class $\kappa_j$}Each tautological class $\kappa_j$ can be identified with an $\Sp(H)$-equivariant homomorphism $H_1(\u^\geom_{g,n})\to H$. By Theorem \ref{ugn ab}, we have an $\Sp(H)$-equivariant isomorphism
$$
H_1(\u^\geom_{g,n})\cong \Lambda^3_0H\oplus \bigoplus_{j=1}^nH_j.
$$
Denote the projection 
$$
\Lambda^3_0H\oplus \bigoplus_{j=1}^nH_j\to H_j
$$
onto the $j$th copy of $H$ by $p_j$. Let $\H_\Q = \H_\Z\otimes \Q$. 
\begin{theorem}[{\cite[Prop.~12.1, Cor.~12.6]{hain_ration}}]
When $g\geq 3$, $n\geq 0$, and $m\geq1$,  there is an isomorphism 
$$
H^1(\M_{g,n}[m], \H_\Q)\cong \Q\kappa_1\oplus \Q\kappa_2\oplus \cdots \oplus \Q\kappa_n
$$ and the class $\kappa_j/(2g-2)$ in $H^1(\M_{g,n}[m], \H_\Q)$ corresponds to 
the projection $p_j$ under the isomorphism
$$
H^1(\M_{g,n}[m], \H_\Q)\cong \Hom^\cts_{\Sp(H)}(H_1(\u^\geom_{g,n}), H).
$$
\end{theorem}

%%%%%%%%%%%%%%%%%%%%%%%%%%%%%%%%%%%%%%%%%%%%%%%%%%%%%%%%
\subsection{The class $\kappa^\hyp_j$}  A similar result also holds for the universal hyperelliptic curves. By Theorem \ref{rel compl cohom iso}, we have a natural  isomorphism
$$
H^1(\cH_{g,n}[m], \H_\Q)\cong \Hom^\cts_{\Sp(H)}(H_1(\v^\geom_{g,n}[m]), H).
$$
The following result together with Proposition \ref{tanaka's computation} allows us to determine $H^1(\cH_{g,n}, \H_\Q)$.
\begin{proposition}\label{ab of unip radical}
If $g\geq 2$, $n\geq 0$, and $m\geq 1$, then there is an  $\Sp(H)$-equivariant isomorphism
$$
H_1(\v^\geom_{g,n}[m])\cong \bigoplus_{j=1}^nH_j \oplus H_1(\v^\geom_g[m]).
$$
\end{proposition}
\begin{proof}
Taking the abelianization of the second diagram in Proposition \ref{commu diagram for lie alg}, we have
$$
\xymatrix@R=1em@C=2em{
\bigoplus_{j=1}^nH_j\ar[r]  \ar@{=}[d]      & H_1(\widehat\v^\geom_{g,n}[m])\ar[r]\ar[d] & H_1(\v^\geom_g[m])\ar[r]\ar[d] ^{\mathrm{diag}} &0\\
\bigoplus_{j=1}^nH_j\ar[r]                    &\bigoplus_{j=1}^nH_1(\v^{\geom(j)}_{g,1}[m])\ar[r]    & (H_1(\v^\geom_{g}[m]))^n\ar[r]&0.
}
$$
Since the adjoint map $\p\to \v^\geom_{g,1}[m]\to W_{-1}\Der \p$ is injective and a morphism of MHS, it follows that the 
$H=\Gr^W_{-1}\p\to \Gr^W_{-1}\v^\geom_{g,n}[m]\to \Gr^W_{-1}\Der\p$ is injective. Note that this adjoint action on $\Gr^W_{-1}$ agrees with the composition $H\to H_1(\v^\geom_{g,1}[m])\to H_1(W_{-1}\Der\p)\to \Gr^W_{-1}\Der \p $. Therefore, the map $H\to H_1(\v^\geom_{g,1}[m])$ is injective. Hence the bottom row of the diagram above is exact at left and so is the top row. Since $H_1(\widehat\v^\geom_{g,n}[m])$ is a direct product of finite dimensional irreducible $\Sp(H)$-modules, we have a desired $\Sp(H)$-splitting. 
\end{proof}
Therefore, for our case, we have the following result. 
\begin{proposition}
If $g\geq 2$ and  $n\geq 0$,  then there is an isomorphism
$$
H^1(\cH_{g,n}, \H_\Q) \cong\Q\kappa^\hyp_1\oplus \Q\kappa^\hyp_2\oplus\cdots\oplus\kappa^\hyp_n.
$$
\end{proposition}
\begin{proof}
The morphism $\cH_{g,n}\to \M_{g,n}$ yields the commutative diagram
$$
\xymatrix@R=1em@C=2em{
H^1(\M_{g,n},\H_\Q)\ar[r]^-\cong\ar[d]& \Hom^\cts_{\Sp(H)}(H_1(\u^\geom_{g,n}), H)\ar[d]\\
H^1(\cH_{g,n},\H_\Q)\ar[r]^-\cong & \Hom^\cts_{\Sp(H)}(H_1(\v^\geom_{g,n}), H),
}
$$
 where the horizontal maps are isomorphisms by Theorem \ref{rel compl cohom iso}.  Since the map $H_1(\v^\geom_{g,n})\to H_1(\u^\geom_{g,n})$ induced by relative completion is a morphism of MHS and $H_1(\v^\geom_g)$ is pure of weight $-2$ by Proposition \ref{pure weight -2}, the image of $H_1(\v^\geom_{g,n})$ is $\bigoplus_{j=1}^nH_j$. Hence the homorphism $p_j$ pulls back to the projection onto the $j$th copy of $H$
$$
H_1(\v^\geom_{g,n})\cong \bigoplus_{j=1}^nH_j \oplus H_1(\v^\geom_g)\to H_j. 
$$
By Propostion \ref{tanaka's computation} and Proposition \ref{ab of unip radical}, there is an isomorphism
$$
\Hom^\cts_{\Sp(H)}(H_1(\v^\geom_{g,n}), H)\cong \Q^n,
$$ 
where the pullbacks of the projections $p_j$ form a basis. Therefore, we have an isomorphism 
$$
H^1(\cH_{g,n}, \H_\Q) \cong \bigoplus_{j=1}^n\Q\kappa^\hyp_j.
$$
\end{proof}

%%%%%%%%%%%%%%%%%%%%%%%%%%%%%%%%%%%%%%%%%%%%%%%%%%%%%%%%%%%%%
\subsection{The $\ell$-adic case} The $\ell$-adic and topological Abel-Jacobi maps are compatible via standard comparison theorems (see \cite[\S 4]{hain_matsu}). This means that for a section $x$ of  $f:C\to \X_{/k}$ the image of the cycle $(2g-2)x-\omega_f$ under the $\ell$-adic Abel-Jacobi map corresponds to $\kappa_x$ under the comparison isomorphism
$$
H^1_\et(\X_{/k}\otimes \bar k, \H_\Zl)\cong H^1_\et(\X_{/k}\otimes \C, \H_\Zl)\cong H^1(\X, \H_\Z)\otimes\Zl.
$$
%The characteristic class of a section $x$ of  $\pi:\cC_{\cH_{g,n/\C}[m]}\to \cH_{g,n/\C}[m]$ in $H^1_\et(\cH_{g,n/\C}[m], \H_\Ql)$ corresponds to $\kappa_x$ under the comparison isomorphism
%$$
%H^1_\et(\cH_{g,n/\C}[m], \H_\Ql)\cong H^1(\cH_{g,n}[m], \H_\Q)\otimes \Ql.
%$$
Therefore, we denote the image of $(2g-2)x-\omega_f$ under the $\ell$-adic Abel-Jacobi map by $\kappa_x$ as well. 
%%%%%%%%%%%%%%%%%%%%%%%%%%%%%%%%%%%%%%%%%%%%%%%%%%%%%%%%%%%%%%%%%%%%%%%%%%%%%%%%%%%%%%%%%%%%%%%%%%%%%%%%%%%%%%%%
\section{The proofs of main results}\label{proofs}
In \cite[\S 8]{wat_sec}, associated to the $n$th power $\cC^n_{\cH_{g/k}}$ of $\cC_{\cH_{g/k}}$, a 4-step pronilpotent Lie aglebra, denoted by  $\h_{g,n}$,  was introduced. It was constructed from the Lie algebra $\widehat\v_{g,n}$ of the prounipotent radical of the weighted completion of $\pi_1(\cC^n_{\cH_{g/k}})$, where $k$ is a field of characteristic zero such that the image of $\ell$-adic cyclotomic character $\chi_\ell:G_k\to \Zl^\times$ is infinite.  The weight filtration on $\widehat\v_{g,n}$ used in \cite{wat_sec} was produced by weighted completion. For our purpose, the Lie algebra $\widehat\v_{g,n}$ is replaced with $\widehat\v^\geom_{g,n}$ equipped with the natural weight filtration defined by the Hodge theory in this paper. Together with Proposition \ref{commu diagram for lie alg}, the facts that $H_1(\v^\geom_g)$ is pure of weight $-2$ and the functor $\Gr^W_\bullet$ is exact on the category of MHSs allow us to use the analogous results in \cite[\S 8]{wat_sec} for our case replacing the weighted completions with the relative completions.  \\
\indent
Since the weight filtrations on the relative completions in this paper are defined over $\Q$, the results in this section holds for the continuous relative completions over $\Ql$, where $\ell$ is a prime number. 

%%%%%%%%%%%%%%%%%%%%%%%%%%%%%%%%%%%%%%%%%%%%%%%%%%%%%%%%
\subsection{4-step nilpotent Lie algebras} 
For $g \geq 2$ and $n\geq 0$, define $\h^\geom_{g,n}$, $\h^\geom_{\cC'_{g,n}}$, $\h^\geom_{\cC_{g,n}}$, and $\widehat\h^\geom_{g,n}$ as the 4-step graded Lie algebras
$$
\h^\geom_{g,n}= \Gr^W_\bullet(\v^\geom_{g,n}/W_{-5}),\,\,\,\,\h^\geom_{\cC'_{g,n}}= \Gr^W_\bullet(\v^\geom_{\cC'_{g,n}}/W_{-5}),
$$
$$
\h^\geom_{\cC_{g,n}}= \Gr^W_\bullet(\v^\geom_{\cC_{g,n}}/W_{-5}),\,\,\text{ and }\,\,\widehat\h^\geom_{g,n}= \Gr^W_\bullet(\widehat\v^\geom_{g,n}/W_{-5}),
$$
respectively. In \cite{wat_sec}, the analogue of  $\widehat\h^\geom_{g,n}$ constructed using the weighted completion of $\pi_1(\cC^n_{\cH_{g/k}})$ was denoted by $\h_{g,n}$. 
By Proposition \ref{commu diagram for lie alg}, there are the exact sequences of MHSs:
$$
0\to\p^o\to\v^\geom_{g, n+1} \to \v^\geom_{g,n}\to0,
$$
$$
0\to\p'\to\v^\geom_{\cC'_{g,n}}\to\v^\geom_{g,n}\to 0,
$$
$$
0\to\p\to\v^\geom_{\cC_{g,n}}\to\v^\geom_{g,n}\to 0,
$$
and 
$$
0\to \p\to \widehat\v^\geom_{g,n+1}\to \widehat\v^\geom_{g,n}\to0.
$$
Since the functor $\Gr^W_\bullet$ is exact, we obtain the exact sequences of graded Lie algebras
$$
0\to\Gr^W_\bullet(\p^o/W_{-5})\to\h^\geom_{g, n+1}\overset{\beta_n^o}\to\h^\geom_{g,n}\to0,
$$
$$
0\to\Gr^W_\bullet(\p'/W_{-5})\to\h^\geom_{\cC'_{g,n}}\overset{\beta'_n}\to\h^\geom_{g,n}\to 0,
$$
$$
0\to\Gr^W_\bullet(\p/W_{-5})\to\h^\geom_{\cC_{g,n}}\overset{\bar\beta_n}\to\h^\geom_{g,n}\to 0,
$$
and
$$
0\to \Gr^W_\bullet (\p/W_{-5})\to \widehat\h^\geom_{g, n+1}\overset{\beta_n}\to \widehat\h^\geom_{g,n}\to 0.
$$
Denote the projections $\h^\geom_{g, n+1}\to\h^\geom_{g,n}$, $\h^\geom_{\cC'_{g,n}}\to\h^\geom_{g,n}$, $\h^\geom_{\cC_{g,n}}\to\h^\geom_{g,n}$, and  $\widehat\h^\geom_{g, n+1}\to \widehat\h^\geom_{g,n}$, by $\beta^o_n$, $\beta'_n$, $\bar\beta_n$, and $\beta_n$, respectively.  In this paper, the $\Gr^W_{-1}$ and $\Gr^W_{-2}$ parts  play an essential role in our main results. Their descriptions immediately follow from  Corollary \ref{surface weight -1 -2}, Corollary \ref{open surface weight -1 -2}, Corollary \ref{pure braid weight -1 and -2}, and Corollary \ref{abelian weight -2}.
\begin{proposition}\label{weight -1-2 parts}
%\textcolor{red}{Describe $\h^\geom_{g,n}$, $\h^\geom_{\cC'g, n}$, $\h^\geom_{\cC g, n}$, and $\widehat\h^\geom_{g,n}$}
Suppose that $g \geq 2$ and $n\geq 0$. There are $\Sp(H)$-equivariant isomorphisms
$$
\Gr^W_r\h^\geom_{g,n}=\begin{cases}
					\bigoplus_{j=1}^nH_j &\text{ for } r =-1,\\
					\bigoplus_{j=1}^n\Lambda^2_0H_j\oplus \bigoplus_{1\leq i<j\leq n} \Q(1)_{ij} \oplus 									H_1(\v^\geom_g) &\text{ for } r=-2,
					\end{cases}
$$
$$
\Gr^W_r\h^\geom_{\cC'_{g,n}}=\begin{cases}
					\bigoplus_{j=0}^nH_j &\text{ for } r =-1,\\
					\bigoplus_{j=0}^n\Lambda^2_0H_j\oplus\Q(1)_{01}\oplus\bigoplus_{1\leq i<j\leq n} \Q(1)_{ij} \oplus 									H_1(\v^\geom_g) &\text{ for } r=-2,
					\end{cases}
$$
$$
\Gr^W_r\h^\geom_{\cC_{g,n}}=\begin{cases}
					\bigoplus_{j=0}^nH_j &\text{ for } r =-1,\\
					\bigoplus_{j=0}^n\Lambda^2_0H_j\oplus\bigoplus_{1\leq i<j\leq n} \Q(1)_{ij} \oplus 									H_1(\v^\geom_g) &\text{ for } r=-2,
					\end{cases}
$$
and
$$
\Gr^W_r\widehat\h^\geom_{g,n}=\begin{cases}
					\bigoplus_{j=1}^nH_j &\text{ for } r =-1,\\
					\bigoplus_{j=1}^n\Lambda^2_0H_j \oplus H_1(\v^\geom_g) &\text{ for } r=-2.
					\end{cases}
$$
\end{proposition}

%%%%%%%%%%%%%%%%%%%%%%%%%%%%%%%%%%%%%%%%%%%%%%%%%%%%%%%%
\subsection{Sections of $\bar\beta_n$ and $\beta_n$} Here, we restate the results from \cite[\S 8]{wat_sec} needed for our case with relative completions.  
\begin{proposition}[{\cite[Prop. 8.2]{wat_sec}}]\label{section from pi to nth power}
Suppose that $g\geq 2$ and $n\geq 0$. Every section of $\pi:\cC_{\cH_{g,n/\C}}\to \cH_{g,n/\C}$ induces an $\Sp(H)$-equivariant graded Lie algebra section of $\bar\beta_n$, which descends to a section of $\beta_n$.
\end{proposition}
\begin{proposition}With the same notation as above, there is a bijection between the set of the $\Sp(H)$-equivariant graded Lie algebra sections of $\bar\beta_n$ and that of $\beta_n$. 
\end{proposition}
\begin{proof}
Consider the commutative diagram
$$
\xymatrix@R=1em@C=2em{
0\ar[r]&\Gr^W_\bullet(\p/W_{-5})\ar[r]\ar@{=}[d]&\h^\geom_{\cC_{g,n}}\ar[r]^{\bar\beta_n}\ar[d]&\h^\geom_{g,n}\ar[r]\ar[d]& 0\\
0\ar[r]& \Gr^W_\bullet (\p/W_{-5})\ar[r]& \widehat\h^\geom_{g, n+1}\ar[r]^{\beta_n}& \widehat\h^\geom_{g,n}\ar[r]& 0.
}
$$
The right square is a pullback square, and hence each section of $\beta_n$ induces a section of $\bar\beta_n$. 
This gives the inverse of the association given in Proposition \ref{section from pi to nth power}.
\end{proof}
By Proposition \ref{weight -1-2 parts}, the weight $-1$ part of $\Gr^W_{\bullet}\beta_n$ is given by
$$
0\to H_0\to H_0\oplus \bigoplus_{j=1}^nH_j\overset{\Gr^W_{-1}\beta_n}\to \bigoplus_{j=1}^nH_j\to 0.
$$
Therefore, by Schur's Lemma,  each section $\zeta$ of $\beta_n$ on the $\Gr^W_{-1}$ part is given by
$$
\Gr^W_{-1}\zeta: (u_1, \ldots, u_n)\mapsto (\sum_{j=1}^na_ju_j, u_1, \ldots, u_n).
$$
In fact, each section of $\beta_n$ is determined by its effect on  the $\Gr^W_{-1}$ part.
\begin{theorem}[{\cite[Thm.~8.4, Cor.~8.6]{wat_sec}}]\label{weight -1 action}
If $g \geq 3$ and $n\geq 0$, then there are exactly $2n$ $\Sp(H)$-equivariant graded Lie algebra sections of $\bar\beta_n$ and $\beta_n$  given by $\zeta_1^\pm,\ldots, \zeta_n^\pm$, where $\Gr^W_{-1}\zeta_j^\pm$ is given by
$$
\Gr^W_{-1}\zeta_j^\pm: (u_1, \ldots, u_n)\mapsto (\pm u_j, u_1, \ldots, u_n).
$$
In particular, $\bar \beta_{0}$ and $\beta_0$ have no sections. 
\end{theorem}

\begin{remark}It is necessary to consider the first 4 steps of the Lie algebras of the completions. A key result used in determining the sections of $\beta_n$ in \cite{wat_sec} is the fact that the bracket of $\Gr^W_{-2}\v_g\subset \Gr^W_{-2}\h_{g,n}$  maps nontrivially and diagonally into $\bigoplus_{j=1}\Gr^W_{-4}\p_j$ in $\Gr^W_{-4}\h_{g,n}$. 
\end{remark}
%Recall that the sections $s_j$ of $\pi: \cC_{\cH_{g,n/k}}\to \cH_{g,n/k}$ are the pullback of the tautological sections $x_j$ of $\cC_{g,n/k}\to \M_{g,n/k}$ and that by composing with the fiber-preserving hyperelliptic involution $J$ of $\pi$, we have the hyperelliptic conjugates $J\circ s_1, \ldots, J\circ s_n$.
Then we have the following correspondence for the sections of $\pi: \cC_{\cH_{g,n/\C}}\to \cH_{g,n/\C}$ and $\bar\beta_n$. 
 \begin{proposition}\label{geo sect induce lie sect}
Suppose that $g\geq 3$ and $n\geq 1$.  For $j=1, \ldots, n$,  the sections $s_j$ and $J\circ s_j$ induce $\zeta_j^+$ and $\zeta_j^{-}$, respectively. 
\end{proposition}
\begin{proof}
Denote the tautological section of $\cC_{g,1/\C}\to \M_{g,1/\C}$ by $\xi$. Let $j\in \{1, \ldots, n\}$. There is a unique morphism $\psi_j: \cH_{g,n/\C}\to \M_{g,1/\C}$ given by taking the geometric point $[C; \bar x_1, \ldots, \bar x_n]\mapsto [C; \bar x_j]$. Then,  the section $s_j$ of $\pi: \cC_{\cH_{g,n/\C}}\to \cH_{g,n/\C}$ can be obtained as the pullback of $\xi$ along $\psi_j$. 
  By Theorem \ref{weight -1 action}, the section $s_j$ induces a section $\zeta$ of $\bar\beta_n$ that is equal to $\zeta_l^+$ or $\zeta_l^-$ for some $1\leq \ell \leq n$. By the naturality of relative completion, the morphism $\psi_j$ produces the commutative diagram
$$
\xymatrix@R=1em@C=2em{
0\ar[r]&\Gr^W_\bullet(\p/W_{-5})\ar[r]\ar@{=}[d]&\Gr^W_\bullet(\u^\geom_{\cC_{g,1}}/W_{-5})\ar[r]&\Gr^W_\bullet(\u^\geom_{g,1}/W_{-5})\ar@/_1pc/[l]_-{d\xi}\ar[r]&0\\
0\ar[r]&\Gr^W_\bullet(\p/W_{-5})\ar[r]&\h^\geom_{\cC_{g,n}}\ar[u]\ar[r]^{\bar\beta_n}&\h^\geom_{g,n}\ar@/^1pc/[l]^-{\zeta}\ar[u]\ar[r]&0,
}
$$
where $d\xi$ be the graded Lie algebra section of $\Gr^W_\bullet(\u^\geom_{\cC_{g,1}}/W_{-5})\to \Gr^W_\bullet(\u^\geom_{g,1}/W_{-5})$ induced by the tautological section $\xi$. Since $s_j$ is the pullback of $\xi$ along $\psi_j$ and $\zeta$ is induced by $s_j$, in weight $-1$, we have the commutative diagram
$$
\xymatrix@R=1em@C=2em{
0\ar[r]&H_0\ar[r]\ar@{=}[d]&\Lambda^3_0H\oplus H_0\oplus H_1\ar[r]&\Lambda^3_0H\oplus H_1\ar[r]\ar@/_1pc/[l]_-{\Gr^W_{-1}d\xi}&0\\
0\ar[r]&H_0\ar[r]&H_0\oplus \bigoplus_{j=1}^nH_j\ar[r]\ar[u]&\bigoplus_{j=1}^nH_j\ar@/^1pc/[l]^-{\Gr^W_{-1}\zeta}\ar[r]\ar[u]\ar[r]&0.
}
$$
   By \cite[Prop.~10.8]{hain_ration}, we have
$$
\Gr^W_{-1}d\xi: (v, u_1)\mapsto (v, u_1, u_1),
$$
where $v$ is in $\Lambda^3_0H$ and $u_1$ in $H_1$.  On the other hand, we have 
$$
\Gr^W_{-1}\zeta: (u_1, \ldots, u_n)\mapsto (\pm u_l, u_1,\ldots, u_n).
$$
Since the right-hand vertical map sends $(u_1, \ldots, u_n)$ to $(0, u_j)$, it follows from the commutativity of the diagram that $\pm u_l = u_j$ in $H_0$, and therefore, $\zeta = \zeta^+_j$.   The restriction of the hyperelliptic involution $J$ to the fiber $C$ induces $-\id$ on $H$, and hence $J\circ s_j$ yields $\zeta^-_j$. 
\end{proof}

%%%%%%%%%%%%%%%%%%%%%%%%%%%%%%%%%%%%%%%%%%%%%%%%%%%%%%%%
\subsection{Proof of Theorem 1}
 Let $x$ be a section of $\cC_{\cH_{g,n/\C}}\to \cH_{g,n/\C}$. By pulling back $x$ over $\cH_{g,n/\C}[m]$ with $m\geq1$, we consider it as a section of $\cC_{\cH_{g,n/\C}}[m]\to \cH_{g,n/\C}[m]$. Let $\kappa_x$ be the corresponding class in $H^1(\cH_{g,n}[m], \H_\Z)$ defined by the  Abel-Jacobi map. The section $x$ induces an $\Sp(H)$-equivariant graded Lie algebra section of $\bar \beta_n$. By Theorem \ref{weight -1 action}, it is equal to one of $\zeta^\pm_1, \ldots, \zeta^\pm_n$. Without loss of generality, say it is $\zeta^+_1$. Then by Proposition \ref{geo sect induce lie sect}, we have $\kappa_x = \kappa_1^\hyp$ in $H^1(\cH_{g,n}[m], \H_\Q)$. Let $m\geq 3$.   Then $\cH_{g,n/\C}[m]$ is a smooth quasi-projective scheme over $\C$ and hence $x$ can be defined over a field $k$ that is finitely generated over $\Q$.   Let $\bar k$ be the algebraic closure of $k$ in $\C$ and $G_k$ the absolute Galois group of $k$. Let $\ell$ be a prime number so that the $\ell$-adic cyclotomic character $\chi_\ell:G_k\to \Zl^\times$ has infinite image. 
%For each finite dimensional irreducible $\Sp(H)$-representation $V$, let $V_\Ql$ be $V\otimes \Ql$. It is an irreducible $\Sp(H_\Ql)$-representation.  
There is the exact sequence of algebraic fundamental groups
$$
1\to \pi_1^\alg(\cH_{g,n/\bar k}[m])\to \pi_1^\alg(\cH_{g,n/k}[m])\to G_k\to 1,
$$
which yields a spectral sequence
$$
E^{s, t}_2 = H^s(G_k, H^t(\pi_1^\alg(\cH_{g,n/\bar k}[m]), H_\Ql))\Rightarrow H^{s+t}(\pi_1^\alg(\cH_{g,n/k}[m]), H_\Ql).
$$
This gives an exact sequence
\begin{align*}
0\to H^1(G_k, H^0(\pi_1^\alg(\cH_{g,n/\bar k}[m]), H_\Ql))\to H^{1}(\pi_1^\alg(\cH_{g,n/k}[m]), H_\Ql)\hspace{1in}\\
\to H^0(G_k, H^1(\pi_1^\alg(\cH_{g,n/\bar k}[m]), H_\Ql)).
\end{align*}
Since $H_\Ql$ is an irreducible $\Sp(H_\Ql)$-module and $\rho^\hyp_\ell[m]: \pi_1^\alg(\cH_{g,n/\bar k}[m], \etabar_n)\to \Sp(H_\Ql)$ has Zariski-dense image, we have $H^0(\pi_1^\alg(\cH_{g,n/\bar k}[m]), H_\Ql)=0$. Therefore, it follows that the restriction maps in the commutative diagram 
$$\xymatrix@R=1em@C=2em{
H^{1}(\pi_1^\alg(\cH_{g,n/k}[m]), H_\Ql)\ar[r]\ar[d]^{\cong} &H^{1}(\pi_1^\alg(\cH_{g,n/\bar k}[m]), H_\Ql)^{G_k}\ar[r] \ar[d]^{\cong}&H^{1}(\pi_1^\alg(\cH_{g,n/\bar k}[m]), H_\Ql)\ar[d]^{\cong}\\
H^{1}_\et(\cH_{g,n/k}[m], \H_\Ql)\ar[r] &H^{1}_\et(\cH_{g,n/\bar k}[m], \H_\Ql)^{G_k}\ar[r]&H^{1}_\et(\cH_{g,n/\bar k}[m], \H_\Ql)
}
$$ are injective. The base change morphism $\cH_{g,n/\C}[m] \to\cH_{g,n/\bar k}[m]$ induces an isomorphism $H^{1}(\pi_1^\alg(\cH_{g,n/\bar k}[m]), H_\Ql) \cong H^{1}(\pi_1^\alg(\cH_{g,n/\C}[m]), H_\Ql)$.   Since $x$ and $s_1$ are defined over $k$, we may view the classes $\kappa_x$ and $\kappa^\hyp_1$ as elements in $H^1_\et(\cH_{g,n/k}[m], \H_\Ql)$ given by the $\ell$-adic Abel-Jacobi map.  Since $\kappa_x = \kappa_1^\hyp$ in $H^1(\cH_{g,n/\C}[m], \H_\Ql)$, it then follows that we have $\kappa_x = \kappa_1^\hyp$ in $H^1_\et(\cH_{g,n/k}[m], \H_\Ql)$. 
This allows us to use the proof of \cite[Thm.~1]{wat_sec} to conclude  that $x = s_1$. Here is a brief sketch of the argument. 
Let $J_{\cH_{g,n/k}}[m]\to \cH_{g,n/k}[m]$ be the relative Jacobian associated to $\cC_{\cH_{g,n/k}}[m]\to \cH_{g,n/k}[m]$. By \cite[Cor.~12.4]{hain_ration}, $t:=s_1 -x$ is torsion in $J_{\cH_{g,n/k}}[m](\cH_{g,n/k}[m])$. If $t =0$, then $x =s_1$. Otherwise, $x$ is disjoint from $s_1$, and then there is an induced $k$-morphism $\phi: \cH_{g,n/k}[m]\to \cH_{g,2/k}$ determined by $[C]\mapsto [C; s_1, x]$. The morphism $\phi$ induces a $\GSp(H_\Ql)$-module map $\Gr^{\mathcal{W}}_\bullet d\phi: \Gr^\mathcal{W}_\bullet\v_{g,n}[m]/\mathcal{W}_{-3}\to \Gr^\mathcal{W}_\bullet \v_{g,2}/\mathcal{W}_{-3}$ such that $\Gr^\mathcal{W}_{-1}d\phi: H_\Ql^{\oplus n}\to H_\Ql^{\oplus 2}$ is given by $(u_1, \ldots, u_n)\mapsto (u_1, u_1)$. 
By \cite[Prop.~9.5]{wat_sec}, this is impossible. 
 Therefore, the universal hyperelliptic curve $\cC_{\cH_{g,n/\C}}\to \cH_{g,n/\C}$ admits exactly $2n$ sections consisting of $s_1, \ldots, s_n$ and their hyperelliptic conjugates.

%On the other hand, $x$ induces a section $dx_\ast^\ab$ of $H_1(\v^\geom_{\cC_{g,n}}[m])\to H_1(\v^\geom_{g,n}[m])$ that is a morphism of MHSs. Let $p_0:\Gr^W_{-1}H_1(\v^\geom_{\cC_{g,n}}[m]) = H_0\oplus H_1\oplus\cdots \oplus H_n\to H_0$ be the projection onto the $0$th copy of $H$. 
 %%%%%%%%%%%%%%%%%%%%%%%%%%%%%%%%%%%%%%%%%%%%%%%%%%%%%%%%%%%%%
\subsection{Sections of $\beta'_n$} Here, we will show that removing the section $s_1$ kills the sections $\zeta_1^\pm$ and $\zeta^{-}_j$ for $j= 2, \ldots, n$. 
\begin{proposition}\label{sections for one puncture case}
When $g\geq 3$ and $n\geq 1$, the $\Sp(H)$-equivariant graded Lie algebra sections of $\bar\beta_n$ that lift to a section of $\beta'_n$ are exactly $\zeta_2^+, \ldots, \zeta_n^+$. \end{proposition}
\begin{proof}
Since each section $s_j$ of $\pi$ with  $2\leq j\leq  n$ is disjoint from $s_1$,  it yields a section of $\pi'$, and hence that of $\beta'_n$. Since $\zeta_j^+$ is induced by $s_j$ by Proposition \ref{geo sect induce lie sect}, the section of $\beta'_n$ induced 
by $s_j$ is a lift of $\zeta_j^+$ to $\h^\geom_{\cC'g,n}$. Now, let $s$ be an $\Sp(H)$-equivariant graded Lie algebra section of $\beta'_n$. Then composing with the natural projection $\h^\geom_{\cC'g,n}\to \h^\geom_{\cC g, n}$, we obtain a section of $\bar\beta_n$. By Theorem \ref{weight -1 action}, $\Gr^W_{-1}s$ is given by $(u_1, \ldots, u_n)\mapsto (\pm u_j, u_1, \ldots, u_n)$. Let $p_0: \Gr^W_{-1}\h^\geom_{\cC'g,n}\to H_0$ be the projection onto the $0$th copy of $H$ in  $\bigoplus_{j=0}^nH_j$. Note that $\Gr^W_{-2}\h^\geom_{\cC'g,n}$ contains a copy of $\Q(1)$,  which comes from $\Gr^W_{-2}\p'$. This is the image in $\Gr^W_{-2}\h^\geom_{\cC'g,n}$ of $\Q(1)_{01}$ in $\Gr^W_{-2}\h^\geom_{g,n+1}$. Then the proof of Theorem \ref{no section for punctured family} shows that $p_0\circ \Gr^W_{-1}s(u_1, \ldots, u_n)\not =\pm u_1$. Therefore,  $\zeta^\pm_1$ do not lift to  a section of $\beta'_n$. Next, we will show that $p_0\circ\Gr^W_{-1}s(u_1, \ldots, u_n) \not= -u_j$ for any $2\leq j\leq n$.  Without loss of generality, assume that  $p_0\circ\Gr^W_{-1}s(u_1, \ldots, u_n) = -u_2$. Then, since $\Theta_1 = -\frac{1}{g}\sum_{j > 1}\Theta_{1j}$ in $\Gr^W_{-2}\h^\geom_{g,n}$, we have
\begin{align*}
\Gr^W_{-2}s(\Theta_1) 	=& -\frac{1}{g}\sum_{j>1}\Gr^W_{-2}s(\Theta_{1j})\\
					=& -\frac{1}{g}\sum_{j> 1}\sum_{l=1}^g[\Gr^W_{-1}s(a_l^{(1)}), \Gr^W_{-1}s(b_l^{(j)})])\\
					=& -\frac{1}{g}\sum_{l=1}^g[a_l^{(1)}, -b_l^{(0)}+b_l^{(2)}]-\frac{1}{g}\sum_{j>2}\sum_{l=1}^g[a_l^{(1)}, b_l^{(j)}]\\
					=& \frac{1}{g}\Theta_{01}-\frac{1}{g}\Theta_{12}-\frac{1}{g}\sum_{j>2}\Theta_{1j}\\
					=&\frac{1}{g}\Theta_{01}-\frac{1}{g}\sum_{j>1}\Theta_{1j}.
\end{align*}

On the other hand, in $\Gr^W_{-2}\h^\geom_{\cC'g,n}$, we have 
\begin{align*}
\Gr^W_{-2}s(\Theta_1)      =& \Gr^W_{-2}s(\sum_{l=1}^g[a_l^{(1)}, b_l^{(1)}])\\
					=& \sum_{l=1}^g[\Gr^W_{-1}s(a_l^{(1)}), \Gr^W_{-1}s(b_l^{(1)})]\\
					=& \sum_{l=1}^g[a_l^{(1)},  b_l^{(1)}]\\
					=&\Theta_1\\
					=&-\frac{1}{g}\Theta_{01}-\frac{1}{g}\sum_{j>1}\Theta_{1j}.
\end{align*}
This implies that $\Theta_{01}=0$, which is a contradiction, since the component $\Q(1)_{01}$ in $\Gr^W_{-2}\h^\geom_{\cC'g,n}$ is generated by $\Theta_{01}$. Therefore, $\zeta_2^-$ does not lift to $\h^\geom_{\cC'g,n}$, nor does $\zeta_j^-$ for $3\leq j\leq n$. 
\end{proof}

%%%%%%%%%%%%%%%%%%%%%%%%%%%%%%%%%%%%%%%%%%%%%%%%%%%%%%%%
\subsection{Proof of Theorem 2} It follows immediately from Proposition \ref{sections for one puncture case} that the hyperelliptic conjugates $J\circ s_1, \ldots, J\circ s_n$ of the sections $s_1, \ldots, s_n$ do not give rise to sections of 
the one-punctured universal hyperelliptic curve $\pi':\cC'_{\cH_{g,n/\C}}\to \cH_{g,n/\C}$. Therefore, by Theorem \ref{sections over C}, there are exactly $n-1$ sections of $\pi: \cC_{\cH_{g,n/\C}}\to \cH_{g,n/\C}$ that are disjoint from $s_1$, namely, $s_2, \ldots, s_n$.

%%%%%%%%%%%%%%%%%%%%%%%%%%%%%%%%%%%%%%%%%%%%%%%%%%%%%%%%
\subsection{Nonexistence of sections of $\beta^o_{n}$ } We will show that there are no $\Sp(H)$-equivariant graded Lie algebra sections of $\beta^o_n$.  
%\textcolor{red}{Define $\Theta_i$ and $\Theta_{ij}$. Define $\Q(1)_{ij}$. }
\begin{proposition}\label{no section for punctured family}
If $g\geq 3$ and $n\geq 0$, then $\beta^o_n$ admits no $\Sp(H)$-equivariant graded Lie algebra sections. 
\end{proposition}
\begin{proof}
When $n=0$, $\beta^o_0 =\beta_{\cC 0}$, and so there is no section by Theorem \ref{weight -1 action}. Let $\zeta^o$ be an $\Sp(H)$-equivariant graded Lie algebra section of $\beta^o_n$.  Composing with the natural projection $\h^\geom_{g, n+1}\to \h^\geom_{\cC_{g,n}}$, we obtain an $\Sp(H)$-equivariant graded Lie algebra section of $\bar\beta_n$:
$$
\xymatrix@R=1em@C=2em{
\h^\geom_{g,n+1}\ar[r]^{\beta^o_{n}}\ar[d]&\h^\geom_{g,n}\ar@/^/[l]^{\zeta^o}\ar@{=}[d]\\
\h^\geom_{\cC g, n}\ar[r]_{\bar\beta_n}& \h^\geom_{g,n}.
}
$$
By Theorem \ref{weight -1 action}, $\Gr^W_{-1}\zeta^o$ is equal to one of the $\Gr^W_{-1}\zeta_j^\pm$. Without loss of generality, we may assume that $\Gr^W_{-1}\zeta^o =\Gr^W_{-1}\zeta_1^+$ or $\Gr^W_{-1}\zeta_1^-$:
$$
\Gr^W_{-1}\zeta^o: (u_1, \ldots, u_n)\mapsto(\pm u_1, u_1, \ldots, u_n). 
$$
When $n=1$, we have
$$
\Gr^W_{-2}\beta^o_1: \bigoplus_{j=0}^1\Lambda^2_0H_j\oplus\Q(1)_{01}\oplus H_1(\v^\geom_g)\to \Lambda^2_0H_1\oplus H_1(\v^\geom_g).
$$
Note that $\Theta_1 =\sum_{l=1}^g[a_l^{(1)}, b_l^{(1)}]$ is trivial in $\Gr^W_{-2}\h^\geom_{g,1}$, and so $\Gr^W_{-2}\zeta^o(\Theta_1)=0$. On the other hand, we have
\begin{align*} 
\Gr^W_{-2}\zeta^o(\Theta_1)      =& \Gr^W_{-2}\zeta^o(\sum_{l=1}^g[a_l^{(1)}, b_l^{(1)}])\\
					=& \sum_{l=1}^g[\Gr^W_{-1}\zeta^o(a_l^{(1)}), \Gr^W_{-1}\zeta^o(b_l^{(1)})]\\
					=& \sum_{l=1}^g[\pm a_l^{(0)} +a_l^{(1)}, \pm b_l^{(0)} + b_l^{(1)}]\\
					=&\Theta_0 +\Theta_1 \pm 2\Theta_{01}\\
					=&-\frac{1}{g}\Theta_{01} -\frac{1}{g}\Theta_{01}\pm 2\Theta_{01}\\
					=&\frac{-2 \pm2g}{g}\Theta_{01},
\end{align*}
which is nontrivial for $g\geq 3$. Therefore, we get a contradiction. 
Now assume that $n \geq 2$.  Then on $\Gr^W_{-2}$ we have
$$\xymatrix@R=1em@C=2em{
0\ar[r]& \Lambda^2_0H_0\oplus \bigoplus_{j=1}^n\Q(1)_{0j}\ar[r]& \bigoplus_{j=0}^n\Lambda^2_0H_j \oplus\bigoplus_{0\leq i<j\leq n}\Q(1)_{ij}\oplus H_1(\v^\geom_g)\ar[r]^-{\Gr^W_{-2}\beta^o_n}&&&\\
&&\bigoplus_{j=1}^n\Lambda^2_0H_j\oplus \bigoplus_{1\leq i<j\leq n}\Q(1)_{ij}\oplus H_1(\v^\geom_g)\ar[r]& 0.
}$$
For $0\leq i<j\leq n$, denote the projection onto the $ij$th copy of $\Q(1)$ by $q_{ij}$: $\Gr^W_{-2}\h^\geom_{g, n+1}\to \Q(1)_{ij}$. 
Since $\zeta^o$ is a graded Lie algebra section, we have
\begin{align*}
\Gr^W_{-2}\zeta^o(\Theta_1)      =&\sum_{l=1}^g[\Gr^W_{-1}\zeta^o(a_l^{(1)}), \Gr^W_{-1}\zeta^o(b_l^{(1)})]\\
					=&\sum_{l=1}^g[\pm a_l^{(0)}+a_l^{(1)}, \pm b_l^{(0)} + b_l^{(1)}]\\
				        =&\Theta_0 +\Theta_1 \pm 2\Theta_{01}\\
				        =&-\frac{1}{g}\sum_{j\not=0}\Theta_{0j}-\frac{1}{g}\sum_{j\not= 1}\Theta_{1j} \pm 2\Theta_{01}\\
				        =&\frac{-2\pm2g}{g}\Theta_{01} -\frac{1}{g}\sum_{j>1}(\Theta_{0j} +\Theta_{1j}).		
\end{align*}
Therefore, we have $(q_{01}\circ \Gr^W_{-2}\zeta^o)(\Theta_1) = \frac{-2\pm2g}{g}\Theta_{01}$. On the other hand, in $\Gr^W_{-2}\h^\geom_{g,n}$, $\Theta_1 = -\frac{1}{g}\sum_{j > 1}\Theta_{1j}$. Therefore, we have
\begin{align*}
\Gr^W_{-2}\zeta^o(\Theta_1) 	=& -\frac{1}{g}\sum_{j>1}\Gr^W_{-2}\zeta^o(\Theta_{1j})\\
					=& -\frac{1}{g}\sum_{j> 1}\sum_{l=1}^g[\Gr^W_{-1}\zeta^o(a_l^{(1)}), \Gr^W_{-1}\zeta^o(b_l^{(j)})])\\
					=& -\frac{1}{g}\sum_{j>1}\sum_{l=1}^g[\pm a_l^{(0)}+a_l^{(1)}, b_l^{(j)}]\\
					=& -\frac{1}{g}\sum_{j>1}(\pm \Theta_{0j} +\Theta_{1j}).
\end{align*}
Hence, $(q_{01}\circ \Gr^W_{-2}\zeta^o)(\Theta_1) =0$, which is a contradiction. Therefore, our claim follows. 
\end{proof}
\subsection{Proof of Theorem 3}
\begin{lemma}\label{weight fil on ab splits}
Suppose that $g\geq 2$ and $n\geq 1$. 
The exact sequence
$$
H_1(\p_{g,n})\to H_1(\v^\geom_{g,n})\to H_1(\v^\geom_g)\to 0.
$$
is left exact. Moreover, the weight filtration on $H_1(\v^\geom_{g,n})$ canonically splits over $\Q$ and hence there is a natural $\Sp(H)$-equivariant isomorphism of MHS
$$
H_1(\v^\geom_{g,n}) \cong H_1(\p_{g,n})\oplus H_1(\v^\geom_g).
$$
\end{lemma}
\begin{proof}
Consider the exact sequence
$$
H_1(\p_{g,n})\to H_1(\v^\geom_{g,n})\to H_1(\v^\geom_g)\to 0
$$
induced from the exact sequence $0\to\p_{g,n}\to \v^\geom_{g,n}\to \v^\geom_g\to 0$. 
The center-freeness of $\p_{g,n}$ implies that the inner adjoint map $\p_{g,n}\to\v^\geom_{g,n}\to W_{-1}\Der \p_{g,n}$ is injective. Since $\Gr^W_\bullet$ is an exact functor and $H_1(\p_{g,n})$ is pure of weight $-1$ by Proposition \ref{ab of pgn}, the composition homomorphism 
$$
H_1(\p_{g,n})=\Gr^W_{-1}\p_{g,n}\to \Gr^W_{-1}\v^\geom_{g,n}\to \Gr^W_{-1}\Der \p_{g,n}
$$
is an injection. This composition agrees with 
$$
H_1(\p_{g,n})\to H_1(\v^\geom_{g,n})\to H_1(W_{-1}\Der \p_{g,n})\to \Gr^W_{-1}H_1(W_{-1}\Der \p_{g,n})\cong \Gr^W_{-1}\Der\p_{g,n}.
$$
Therefore, the map $H_1(\p_{g,n})\to H_1(\v^\geom_{g,n})$ is injective. 
On the other hand, we have the surjection $q: \v^\geom_{g,n}\to \Gr^W_{-1}\v^\geom_{g,n} = \Gr^W_{-1}H_1(\v^\geom_{g,n})$. Since $H_1(\v^\geom_g)$ is pure of weight $-2$ by Proposition \ref{pure weight -2}, it follows that there is a canonical isomorphism
$ H_1(\p_{g,n})=\Gr^W_{-1}H_1(\p_{g,n})\cong\Gr^W_{-1}H_1(\v^\geom_{g,n})$ of MHSs. Hence $q$ induces the  projection $H_1(\v^\geom_{g,n})\to H_1(\p_{g,n})$, which gives the canonical splitting of the weight filtration on $H_1(\v^\geom_{g,n})$ over $\Q$:
$$
H_1(\v^\geom_{g,n})\cong\Gr^W_\bullet H_1(\v^\geom_{g,n}) = H_1(\p_{g,n})\oplus H_1(\v^\geom_g).
$$
\end{proof}

\indent Now, consider the homotopy exact sequence 
 \begin{equation}\label{alg homotopy seq}
 1\to \pi_1^\alg(C^o, \bar x_0)\to \pi_1^\alg(\cH_{g,n+1/\C}, \etabar_{n+1})\to \pi^\alg_1(\cH_{g,n/\C}, \etabar_n)\to 1,
 \end{equation}
 which can be identified with the profinite Birman exact sequence (\ref{pro hyp birman seq}) for the hyperelliptic mapping class groups. Suppose that the natural projection $\widehat{\pi^o_\ast}: \pi_1^\alg(\cH_{g,n+1/\C}, \etabar_{n+1})\to \pi^\alg_1(\cH_{g,n/\C}, \etabar_n)$ admits a section $\gamma$. The continuous relative completion of $\pi_1^\alg(\cH_{g,n/\C}, \etabar_n)$ with respect to the homomorphism $\rho^\hyp_\ell$ defined in \S \ref{cont hyp rep} is given by $\cD^\geom_{g,n/\Ql}:=\cD^\geom_{g,n}\otimes \Ql$ and hence its Lie algebra is given by $\d^\geom_{g,n/\Ql}:=\d^\geom_{g,n}\otimes \Ql$. Similarly its prounipotent radical is given by $\cV^\geom_{g,n/\Ql}:=\cV^\geom_{g,n}\otimes \Ql$ and its Lie algebra is given by $\v^\geom_{g,n/\Ql}:=\v^\geom_{g,n}\otimes\Ql$.  Since the weight filtration on $\d^\geom_{g,n}$ is defined over $\Q$, it lifts to $\d^\geom_{g,n/\Ql}$. Applying continuous relative completion to $\widehat{\pi^o_\ast}$, we obtain the induced Lie algebra surjection $d\pi^o_{\ast\Ql}:= d\pi^o_\ast\otimes\Ql: \v^\geom_{g,n+1/\Ql}\to \v^\geom_{g,n/\Ql}$. Fix a splitting $\cD^\geom_{g,n+1}\cong\cV^\geom_{g,n+1}\rtimes\Sp(H)$. Then it defines a splitting $\cD^\geom_{g,n}\cong \cV^\geom_{g,n}\rtimes \Sp(H)$ that is compatible with the projection $\cD^\geom_{g,n+1}\to \cD^\geom_{g,n}$. This defines an $\Sp(H_\Ql)$-module structure on $\v^\geom_{g,n/\Ql}$ and  $d\pi^o_{\ast\Ql}$ is considered as an $\Sp(H_\Ql)$-equivariant Lie algebra surjection.  
 Furthermore, the section $\gamma$ induces a Lie algebra section $d\gamma_\Ql$ of $d\pi^o_{\ast\Ql}$. We see that $d\gamma_\Ql$ induces an $\Sp(H_\Ql)$-equivariant section $d\gamma^\ab_\Ql$ of $d\pi^{o,\ab}_{\ast\Ql}: H_1(\v^\geom_{g,n+1/\Ql})\to H_1(\v^\geom_{g,n/\Ql})$ that is compatible with the projection onto $H_1(\v^\geom_{g/\Ql})$. Denote $\p_{g,n}\otimes \Ql$ by $\p_{g,n/\Ql}$. 
By Lemma \ref{weight fil on ab splits},  there is a natural $\Sp(H_\Ql)$-equivariant isomorphism $H_1(\v^\geom_{g,n/\Ql})\cong H_1(\p_{g,n/\Ql})\oplus H_1(\v^\geom_{g/\Ql})$, and by Proposition \ref{ab of pgn}, we have $H_1(\p_{g,n/\Ql})\cong H_\Ql^{\oplus n}$.  Then $d\pi^o_\ast$ can be described as
$$
H_\Ql^{\oplus n+1}\oplus H_1(\v^\geom_{g/\Ql})\to H_\Ql^{\oplus n}\oplus H_1(\v^\geom_{g/\Ql}): (u_0, u_1, \ldots, u_n, v)\mapsto (u_1, \ldots, u_n, v).
$$
On the other hand, we may consider $d\gamma^\ab_\Ql$ as an $\Sp(H_\Ql)$-module map
$$
H_\Ql^{\oplus n}\oplus H_1(\v^\geom_{g/\Ql})\to H_\Ql^{\oplus n+1}\oplus H_1(\v^\geom_{g/\Ql}).
$$
 Since $d\gamma^\ab_\Ql$ is an $\Sp(H_\Ql)$-equivariant section of $d\pi^{o,\ab}_{\ast\Ql}$, the image of the restriction of $d\gamma^\ab_\Ql$ to $H_\Ql^{\oplus n}$ is contained in  $H_\Ql^{\oplus n+1}$.
%  and the restriction is given by 
% $$
% (u_1, \ldots, u_n)\mapsto (\sum_{j=1}^na_ju_j, u_1, \ldots, u_n).
% $$ 
 Furthermore, it follows from Theorem \ref{rel comp iso for sp} and Proposition \ref{tanaka's computation} that  the restriction of $d\gamma^\ab_\Ql$ to $H_1(\v^\geom_{g/\Ql})$ is the identity map. Hence $d\gamma^\ab_\Ql$ is graded with respect to the weight filtration:
 $$
 d\gamma^\ab_\Ql = \Gr^W_\bullet d\gamma^\ab_\Ql: \Gr^W_\bullet H_1(\v^\geom_{g,n/\Ql})\to \Gr^W_\bullet H_1(\v^\geom_{g, n+1/\Ql}).
 $$
\indent Denote the lower central series of $\v^\geom_{g,n}$ by $L^\bullet\v^\geom_{g,n}$. Since it is defined over $\Q$, it lifts to the lower central series of $\v^\geom_{g,n/\Ql}$. Since $d\gamma_\Ql$ preserves the lower central series, it induces a graded Lie algebra homomorphism $\Gr_L^\bullet d\gamma_\Ql: \Gr_L^\bullet\v^\geom_{g,n/\Ql}\to \Gr_L^\bullet\v^\geom_{g,n+1/\Ql}$. Moreover, the map $d\gamma^\ab_\Ql$ induces a Lie algebra homomorphism of free Lie algebras
$$
\L(d\gamma^\ab_\Ql): \L(H_1(\v^\geom_{g,n/\Ql}))\to \L(H_1(\v^\geom_{g,n+1/\Ql})).
$$
For $n\geq 0$, let $\epsilon_n:H_1(\v^\geom_{g,n/\Ql})=\Gr_L^{1}\v^\geom_{g,n/\Ql}\hookrightarrow \Gr_L^\bullet\v^\geom_{g,n/\Ql}$ be the canonical inclusion. Then $\epsilon_n$ induces an $\Sp(H_\Ql)$-equivariant graded Lie algebra surjection 
$$
\L(\epsilon_n): \L(H_1(\v^\geom_{g,n/\Ql}))\to \Gr_L^\bullet\v^\geom_{g,n/\Ql}
$$
that makes the diagram
$$
\xymatrix@R=2em@C=4em{
\L(H_1(\v^\geom_{g,n+1/\Ql}))/\ker\L(\epsilon_{n+1}) \ar[r]_-\cong^-{\overline{\L(\epsilon_{n+1})}}& \Gr_L^\bullet\v^\geom_{g,n+1/\Ql}\\
\L(H_1(\v^\geom_{g,n/\Ql}))/\ker\L(\epsilon_{n})\ar[u]^{\overline{\L(d\gamma^\ab_\Ql)}}\ar[r]_-\cong^-{\overline{\L(\epsilon_{n})}}& \Gr_L^\bullet\v^\geom_{g,n/\Ql}\ar[u]_{\Gr_L^\bullet d\gamma_\Ql},
}
$$ 
 commute, where  $\overline{\L(d\gamma^\ab_\Ql)}$, $\overline{\L(\epsilon_{n})}$, and $\overline{\L(\epsilon_{n+1})}$ are induced by $\L(d\gamma^\ab_\Ql)$, $\L(\epsilon_{n})$, and $\L(\epsilon_{n+1})$, respectively.
 Note that the free Lie algebra $\L(H_1(\v^\geom_{g,n/\Ql}))$ admits a weight filtration induced by $W_\bullet H_1(\v^\geom_{g,n/\Ql})$. Since $H_1(\v^\geom_{g,n/\Ql})= \Gr^W_\bullet H_1(\v^\geom_{g,n/\Ql})$, the Lie algebra $\Gr_L^\bullet \v^\geom_{g,n/\Ql}$ is also graded by weights: $\Gr_L^\bullet\v^\geom_{g,n/\Ql} =\Gr^W_\bullet\v^\geom_{g,n/\Ql}$.  Since the bracket of $\v^\geom_{g,n}$ is a morphism of MHS, the $j$th component map $\L_j(\epsilon_n): \L_j(H_1(\v^\geom_{g,n/\Ql}))\to \Gr_L^j\v^\geom_{g,n/\Ql}$ is a weight filtration-preserving map, and so is $\overline{\L(\epsilon_{n})}$. Since $d\gamma^\ab_\Ql = \Gr^W_\bullet d\gamma^\ab_\Ql$, the induced map $\L(d\gamma^\ab_\Ql)$ preserves the weight filtration induced on the free Lie algebras, and so does $\overline{\L(d\gamma^\ab_\Ql)}$.  Therefore, $\Gr_L^\bullet d\gamma_\Ql$ preserves the weight filtration, and thus it induces an $\Sp(H_\Ql)$-equivariant graded Lie algebra homomorphism $\widetilde{d\gamma_\Ql}:=\Gr^W_\bullet\Gr^\bullet_L d\gamma_\Ql: \Gr^W_\bullet\v^\geom_{g,n/\Ql}\to \Gr^W_\bullet\v^\geom_{g,n+1/\Ql}$. It is a graded Lie algebra section of $\Gr^W_\bullet d\pi^o_{\ast\Ql}$.   
 It follows from Proposition \ref{no section for punctured family} that there is no such section of $\Gr^W_\bullet d\pi^o_{\ast\Ql}$. Therefore, the sequence (\ref{alg homotopy seq}) does not split. \\
 \subsection{Proof of Corollary \ref{hyp birman seq not split}}
 The profinite completion of the hyperelliptic Birman exact sequence 
 $$
 1\to \pi_1(S_{g,n})\to \Delta_{g,n+1}\to \Delta_{g,n}\to 1
 $$
 can be identified with the sequence (\ref{alg homotopy seq}).  Since a section of $\Delta_{g, n+1}\to \Delta_{g,n}$ induces a section of $\widehat{\Delta_{g,n+1}}\to \widehat{\Delta_{g,n}}$ by profinite completion, the nonsplitting of the sequence (\ref{alg homotopy seq}) implies that the hyperelliptic Birman exact sequence does not split, either. 
 
%%%%%%%%%%%%%%%%%%%%%%%%%%%%%%%%%%%%%%%%%%%%%%%%%%%%%%%%%%%%%%%%%%%%%%%%%%%%%%%%%%%%%%%%%%%%%%%%%%%%%%%%%%%%%%%%

\end{document}